\documentclass[reqno,11pt]{amsart}
\usepackage{geometry}
\usepackage{amsmath,amssymb,mathrsfs,color,paralist}
\usepackage[hidelinks]{hyperref}
\usepackage[T1]{fontenc}
\usepackage{fourier}

\usepackage{calc}
\newtheorem{The}{Theorem}[section]
\newtheorem{Cor}[The]{Corollary}
\newtheorem{Lem}[The]{Lemma}
\newtheorem{Pro}[The]{Proposition}

\newtheorem{Rem}[The]{Remark}
\theoremstyle{remark}

\numberwithin{equation}{section}

\newcommand{\R}{\mathbb{R}}
\newcommand{\Z}{\mathbb{Z}}

\newcommand{\KK}{\mathbf{B}}
\newcommand{\LL}{\mathbf{L}}
\newcommand{\HH}{\mathbf{H}}
\newcommand{\GG}{{\rm G}}
\newcommand{\FF}{\mathcal{F}}
\newcommand{\lam}{\lambda}
\newcommand{\wtM}{\widetilde{\mathfrak{M}}}
\newcommand{\wtp}{\widetilde{p}}
\newcommand{\tmu}{\tilde\mu}
\newcommand{\px}{p^{x,v,0}}
\newcommand{\pxu}{p^{x,v,u}}
\def\leq{\leqslant}
\def\geq{\geqslant}
\newcommand{\xilax}{\xi_x^\lambda}
\newcommand{\dxilax}{\dot{\xi}_x^\lambda}
\newcommand{\vep}{\varepsilon}
\newcommand{\rro}{\mathbf{r_0}}
\newcommand{\alo}{\alpha_0}
\def\du#1{\langle#1\rangle}
\def\cA{\mathcal{A}}

\title[Convergence of solutions of H-J equations]{Convergence of solutions of Hamilton-Jacobi equations depending nonlinearly on the unknown function}
\author{Qinbo Chen}
\address{Dipartimento di Matematica `Tullio Levi-Civita', Universit\`a degli Studi di Padova,  35121 Padova, Italy}
\email{qinbochen1990@gmail.com}
\begin{document}

\begin{abstract}
Motivated by the vanishing contact problem, we study in the present paper the convergence of solutions of Hamilton-Jacobi equations depending nonlinearly on the unknown function. Let  $H(x,p,u)$ be a continuous Hamiltonian which is strictly increasing in $u$, and is convex and coercive in $p$. For each parameter $\lam>0$, we denote
 by $u^\lam$  the unique viscosity solution of the H-J equation
	\[	H\big(x,Du(x),\lam  u(x)\big)=c.\]
 Under quite general assumptions, we prove that $u^\lam$ converges uniformly, as $\lam$ tends to zero, to a specific solution of the critical H-J equation $	H(x,Du(x),0)=c.$  We also characterize the limit solution in terms of Peierls barrier and Mather measures.
\end{abstract}
\maketitle

\section{Introduction}\label{section_Int}
Throughout this paper, we let $M$ be a connected and compact smooth manifold without boundary, equipped with a smooth Riemannian metric $g$, the associated Riemannian distance on $M$ will be denoted by $d$. Let $TM$ and $T^*M$ denote the tangent and cotangent bundles respectively. A point of $TM$ will be denoted by $(x,v)$ with $x\in M$ and $v\in T_xM$, and a point of $T^*M$ by $(x, p)$ with $p\in  T_x^*M$ a linear form on the vector space $T_xM$.   With a slight abuse of notation, we will both denote by $|\cdot|_x$ the norm induced by $g$ on the fiber $T_xM$ and also the dual norm on $T_x^*M$. 

In the sequel, $H(x,p,u):T^*M\times\R\to\R$ will be a given continuous Hamiltonian which is strictly increasing in $u$, and is convex and coercive in the fibers. By adding a parameter $\lam> 0$, we consider the following stationary Hamilton-Jacobi equation 
\begin{align}\label{e1}
	H\big(x,Du(x),\lam  u(x)\big)=c,\quad x\in M.
\end{align}
Here, $c$ is a constant. This equation obeys the comparison principle, and therefore it has a unique viscosity solution $u^\lam: M\to\R$. As $\lam\to 0$, equation \eqref{e1} tends to a critical Hamilton-Jacobi equation 
\begin{align}\label{e2}
	H\big(x,Du(x),0\big)=c,\quad x\in M,
\end{align}
and then we are interested in the asymptotic behavior of the family of solutions $\{u^\lam\}_{\lam>0}$. In the sequel, for abbreviation, we set
 $$\GG(x,p):=H(x,p,0).$$ 
 As is well known, there is only one special \emph{critical value} $c=c(\GG)$ such that the H-J equation $\GG(x, Du(x))=c$ admits  viscosity solutions. 
 Thus it is natural to consider equations \eqref{e1} and \eqref{e2} only for $c=c(\GG)$, i.e.,
\begin{equation}\label{eq_H_lambda}\tag{HJ$_\lambda$}
H\big(x,Du(x),\lam  u(x)\big)=c(\GG),\quad x\in M
\end{equation}
and 
\begin{equation}\label{eq_G}\tag{HJ$_0$}
\GG\big(x,Du(x)\big)=c(\GG),\quad x\in M.
\end{equation}

Note that $H(x,p,\lam u)$ locally uniformly converges to $\GG(x,p)$ as $\lam\to 0$,
by the stability of the viscosity properties,  any accumulation point of the set $\{u^\lam\}_{\lam>0}$ must be a viscosity solution of \eqref{eq_G}. Nevertheless, equation \eqref{eq_G} has infinitely many solutions, and it is not clear  that limits of $\{u^\lam\}_{\lam>0}$ along different subsequences yield the same solution. Our starting point is to study the uniform convergence of the whole family  $\{u^\lam\}_{\lam>0}$.

In the case where $H(x, p, u)=u+\GG(x,p)$, \eqref{eq_H_lambda} is reduced to the so-called \emph{vanishing discount problem}, in view that $\lam>0$ appears as a discount factor in the dynamic programming principle for optimal control problems.
It is also known as the \emph{ergodic approximation} which was  proposed by Lions, Papanicolaou, and Varadhan in \cite{Lions_Papanicolaou_Varadhan1987}.  In the last decade, there have been many new contributions to this  vanishing discount
problem. 
A selection criterion for possible limits of solutions was given in \cite{Gomes2008}, and a partial convergence result was obtained in \cite{Iturriaga_Sanchez-Morgado2011}. 
Recently, the uniform convergence to a unique limit is established by Davini, Fathi, Iturriaga and Zavidovique  \cite{DFIZ2016}, see also \cite{DFIZ_Math_Z_2016} for a  discrete time version. The convergence and the characterization of the limit solution are obtained by taking advantage of Aubry-Mather theory and weak KAM theory.
Subsequently, it has been  generalized under various types of settings. For instance,  the first-order case with Neumann-type boundary condition  \cite{AAIY2016}, the first-order case with the ambient space being non-compact \cite{Ishii_Siconolfi2020}. It has also been generalized to the second-order H-J setting, see  \cite{Mitake_Tran2017, Ishii_Mitake_Tran2017_1, Ishii_Mitake_Tran2017_2}. In the case of weakly coupled system of H-J equations, the vanishing discount problem is investigated in \cite{Davini_Zavidovique2017, Ishii2019vanishing, Ishii_Jin2020} by using a generalization of  Mather (minimizing) measure.  The vanishing discount mean field game problem has been studied in \cite{Cardaliaguet_Porretta2019, gomes2019selection}. For
 the convergence problem from the negative direction (i.e. $\lam<0$), it is discussed in \cite{Davini_Wang2020} under the assumption that constant functions are critical subsolutions. We also refer to \cite{Zavidovique2020} for some degenerate discounted H-J equations.

Observe that the discounted system depends linearly on $u$.
Recently, there has been an interest in generalizing the  asymptotic convergence problem  to those Hamiltonians  depending nonlinearly or implicitly on $u$.  Such extensions  were first discussed in \cite{Gomes_Mitake_Tran2018, CCIZ2019}. For example,
in the case of nonlinear discounted systems, i.e. 
 \begin{equation}\label{separat_case}
 	 H(x,p,u)=f(x,u)+F(x,p),
 \end{equation} 
 the  convergence was obtained by \cite{CCIZ2019} under the following two assumptions:  {\bf(C1)}  $F(x,p)\in C(T^*M)$ is convex and coercive in the fibers; {\bf(C2)} $f(x, u)\in C^1(M\times\R)$ with  $f_u>0$ and 
\begin{equation}\label{loc_lip}
	\limsup_{u\to 0}\frac{|f_u(x,u)-f_u(x,0)|}{|u|}<+\infty, \quad \text{uniformly in~}  x.
\end{equation}
Note that \eqref{separat_case} is special because it is a separate system with $u$ decoupling with $p$. 

For more general $H(x,p,u)$ which is not separate, the work \cite{CCIZ2019}  also gives a convergence  result when the  Hamiltonian  satisfies  the following assumptions: {\bf(D1)}  $H_u\in C(T^*M\times\R)$ with  $H_u>0$. For every $R>0$, $\exists$ $B_R>0$ such that
\begin{equation}\label{old_loc_lip}
	 |H(x,p,u)-H(x,p,0)|\leqslant B_R|u|,\quad\text{for all}~|u|\leqslant R, ~(x,p)\in T^*M,
\end{equation}
\begin{equation}\label{old_rem}
|H_u(x,p,u)-H_u(x,p,0)|\leqslant B_R|u|,\quad \text{for all~} |p|_x\leqslant R, |u|\leqslant R;	
\end{equation}
{\bf(D2)} $\GG(x,p)$ is convex and coercive in the fibers. 
{\bf(D3)}  Either $\frac{\GG(x, p)-c(\GG)}{H_u(x,p,0)}$ is convex and coercive in the fibers, or 
\[\frac{\min_{_{|u|,~ |p|_x\leq R_0}} H_u(x, p, \lam u)}{\max_{_{|u|,~ |p|_x\leq R_0}} H_u(x, p, \lam u)}\longrightarrow 1, \quad\text{~as~} \lam\to 0 \]
for a suitably large constant $R_0$.  
Recently, using the implicit variational principle  (see for instance \cite{Su_Wang_Yan2016, Wang_Wang_Yan2019_Aubry, Wang_Wang_Yan2019_Variational})  of the contact systems, the authors of  \cite{WYZ2020} have introduced a totally new technique to remove the restriction (D3) above for $C^2$ contact Hamiltonians with Tonelli conditions.
Their approach is to establish Radon measures supported on the minimizing curves of the associated contact systems and then verify the convergence to Mather measures. Inspired by their approach, in the present paper we generalize the conditions to the more general case, with a slight losing of dynamical information. 

The goal of the current paper is, following a non-smooth setting analogous to \cite{DFIZ2016, CCIZ2019}, to further investigate the convergence problem for general continuous Hamiltonians with less restrictions. In such a setting,  no Hamiltonian dynamics can necessarily be defined,  so  dynamical systems techniques are not suitable. We instead use nonsmooth analysis and PDE methods.  Also, it is worth noting that,  different from the technique in \cite{CCIZ2019}, we will give an approximation scheme (analogous to \cite[Lemma 3.6]{WYZ2020}) to select certain Mather measures (see Section \ref{section_convergence}), then the convergence of the whole family of solutions follows from this approximation scheme.

\subsection{Setup and main result}
We first describe the setting and state main assumptions.
Suppose that the continuous Hamiltonian $H$ satisfies the following assumptions:
\begin{enumerate}
    \item  [\textbf{(H1)}] The map $ p\mapsto H(x, p, u)$ is convex on $T^*_x M$, for every $(x,u)\in M\times\R$.
    \item  [\textbf{(H2)}] There exists a constant $\rro>0$, such that $H(x,p, -\rro)$ is coercive in the fibers, i.e.,
	\[ \lim_{|p|_x\to \infty} H(x, p, -\rro)=+\infty,\quad \text{uniformly in~}  x\in M.\] 
	\item [\textbf{(H3)}] The map $u\mapsto H(x,p,u)$ is strictly increasing, for every $(x,p)\in T^*M$. 
\end{enumerate}

We point out that the convexity assumption {\bf(H1)} is essential for the  convergence problem, see \cite{Ziliotto_counterexample_2019} for a remarkable counterexample  where the convergence fails without the convexity. 

\begin{Rem}
It is easily seen from the monotonicity assumption {\bf(H3)} that for each $u_0\geq-\rro$, the function $H(x, p, u_0): T^*M\to \R$ is also coercive in the fibers. More precisely, for each $u_0\in[-\rro,+\infty)$,
\begin{equation}\label{coercive_forall_u}
	\lim_{|p|_x\to \infty}H(x, p, u_0)\geq \lim_{|p|_x\to \infty} H(x, p, -\rro)=+\infty, \quad \textup{uniformly in~}  x\in M.
\end{equation}
In particular, the critical Hamiltonian $\GG(x,p)$ is  coercive in the fibers. Also,
we refer the reader to Proposition \ref{equiv_def_H2} for an equivalent statement of assumption {\bf(H2)}.
\end{Rem}

 Assumption {\bf(H3)}, as is known to all, also guarantees the uniqueness of solutions.  See Proposition \ref{uni_bound} later for the existence and uniqueness of  solutions of equation \eqref{eq_H_lambda} for each small parameter $\lam>0$, under  assumptions {\bf(H2)} and {\bf(H3)}. In fact, it is a standard result obtained by  Perron's method and the comparison principle.

 Note that the assumptions above do not require smoothness of $H(x, p, u)$ on the variables $(x, p)$. Nevertheless,
  in order to address the convergence result,  we need the partial derivative with respect to $u$ at $u=0$:
\begin{enumerate} 
\item  [\textbf{(H4)}] The partial derivative $H_u(x,p,0)$  exists,  $H_u(x, p, 0)>0$ and $H_u(x, p, 0)$ $\in C(T^*M)$. Assume further that
\begin{equation}\label{H_locaunif_conv}
	\lim_{u\to 0}\frac{H(x,p,u)-H(x,p,0)}{u}=H_u(x,p,0) ,\quad \text{locally uniformly}
	\footnote{Here, the locally uniform convergence \eqref{H_locaunif_conv}   means that for each compact subset $B\subset T^*M$, we have:  for every $\vep>0$, there exists $\delta>0$ such that 
$\big| [H(x,p,u)-H(x,p,0)]/u-H_u(x,p,0) \big|< \vep$, for all $(x, p)\in B$, $0<|u|< \delta$.}
 \text{~in~}  (x, p)\in T^*M.
\end{equation}
\end{enumerate}

\medskip

\begin{Rem}\label{rem_1}	
Assumption {\bf(H4)} holds naturally if one assumes $H\in C^1(T^*M\times(-\delta, \delta))$ for some constant $\delta>0$ and $H_u(x,p,0)>0$.  Indeed,  it can be easily verified by using Newton-Leibniz formula.

Note that our assumptions do not require the uniformly Lipschitz condition, i.e. $|H_u|\leq C$ with $C$
a positive constant.
\end{Rem}

\begin{Rem}\label{non_neg}
In our setting, 
the partial derivative $ H_u(x,p,u)$ could take the value \textbf{zero} at some points $u\neq 0$. 
For example, \[H(x,p,u)=(u-2)^{2m+1}+|p|_x+V(x),\]
where $V\in C(M)$ and $1\leq m\in\Z$.
Observe that $H(x,p,u)$ is strictly increasing in $u$ (i.e. $H(x,p, u_1)$ $<$ $H(x,p, u_2)$ if $u_1<u_2$), and $ H_u(x,p, 0)>0$. So, it satisfies all assumptions {\bf(H1)--(H4)} above. But  $ H_u(x,p, 2)$ $=0$ for all $(x,p)\in T^*M$. 
\end{Rem}

\begin{Rem}\label{ex_nonlinear}
In particular, for the nonlinear discounted system of the form
\[H(x,p,u)=g(u)+\widehat{H}(x,p).\] To satisfy assumptions {\bf(H1)--(H3)}, it is enough to assume that
 $\widehat{H}(x,p)\in C(T^*M)$ is convex and coercive in the fibers, and $g(u)$ is continuous and strictly increasing. As for assumption {\bf(H4)}, it is equivalent to assume that
 $g(u)$ is differentiable at  $u=0$ with $g'(0)>0$. 
 \end{Rem}

We are now ready to state our first main result:

\begin{The}\label{mainresult1}
Let $H(x,p,u): T^*M\times\R\to\R$  satisfy \mbox{\bf{(H1)--(H4)}}. Then there exists a constant $\alo>0$  such that for each $\lambda\in(0,\alo)$, equation \eqref{eq_H_lambda} has a  unique continuous viscosity solution $u^\lambda$ which is Lipschitz continuous, and 
$u^\lam$ converges uniformly, as $\lam\to 0$, to a Lipschitz function $u^0$ which is exactly a viscosity solution of \eqref{eq_G}.
\end{The}
 
\begin{Rem}
An estimate of  $\alpha_0$ will be given in Section \ref{section_exist_and_uniq}. Actually,  $\alpha_0$ depends only on the  critical Hamiltonian $\GG(x,p)$ and  the value $\rro$, see \eqref{fsgw} for more details.
\end{Rem}

By Remark \ref{ex_nonlinear}, our Theorem \ref{mainresult1} then has the following immediate consequence :

\begin{Cor}
Let $\widehat{H}(x,p): T^*M\to \R$ be a continuous Hamiltonian which is coercive and convex in the momentum. Suppose that $g(u)\in C(\R)$ is a  strictly increasing function, and is differentiable at  $u=0$ with $g'(0)>0$. Then, for each small parameter $\lam>0$, the  equation 
\[g(\lam u(x) )+\widehat{H}(x, Du(x) )=c_0, \quad x\in M,\]
with $c_0$ the critical value of the critical Hamiltonian $g(0)+\widehat{H}(x,p)$,  
 has a  unique continuous viscosity solution $u^\lambda$. Moreover,  $u^\lambda$ converges uniformly to a single critical solution as $\lam\to 0$. 
\end{Cor}
In view of this corollary, the assumptions of the current paper are indeed natural generalizations of that in \cite{DFIZ2016}.

If we replace the coercivity assumption {\bf(H2)} by the fiberwise superlinear growth assumption {\bf(H2*)} below,  then we can introduce the conjugated Lagrangian, and characterize the limit solution $u^0$ by making use of the tools in Aubry-Mather theory.
\begin{enumerate}
	\item [{\bf(H2*)}] There exists a constant $\rro>0$, such that $H(x,p, -\rro)$ is superlinear in the fibers, i.e.,
	\[\lim_{|p|_x\to \infty} \frac{H(x, p,-\rro)}{|p|_x}=+\infty,\qquad \textup{uniformly in} x\in M.\]
\end{enumerate}

Through the Legendre-Fenchel transformation, the Lagrangian $L$  associated to the Hamiltonian $H$ is given by
 \begin{align}\label{def_Lag}
	L(x, v, u)=\sup_{p\in T^*_xM}\{\du{p,v}_x-H(x,p, u)\},\quad  \forall~(x,v,u)\in TM\times\R.
\end{align}
Here, $\du{p,v}_{x}$ denotes the value of the linear form $p\in T_x^*M$ evaluated at $v\in T_xM$. 

Observe that assumptions {\bf(H2*)} and {\bf(H3)} together imply that 
for every  fixed $u_0\in(-\rro,+\infty)$,  the map $p\mapsto H(x, p, u_0)$ is superlinear in $p$, uniformly in $x\in M$, i.e.,
\[\lim_{|p|_x\to \infty} \frac{H(x, p,u)}{|p|_x}\geq \lim_{|p|_x\to \infty} \frac{H(x, p,-\rro)}{|p|_x}=+\infty,\quad \text{uniformly in~} x\in M, ~u\in(-\rro, +\infty).\] 
Therefore,  $L(x,v,u)$ is finite-valued  whenever $u\in(-\rro,+\infty)$.  But $L(x,v,u)$ may take the value $+\infty$ when $u< -\rro$. However, as we will see later, only the information on $TM\times(-\rro, +\infty)$ is needed for our purpose. Therefore, in the remainder of this paper, we only consider the Lagrangian $L(x, v, u): TM\times(-\rro,+\infty)\to\R$.

Here, we also point out that if one assumes that $H(x,p,u_0)$ is superlinear in the fibers for every $u_0\in\R$,
then $L(x,v,u)$ is finite-valued and well defined on the whole space $TM\times\R$.

It is a classical result in convex analysis that 
$L\in C(TM\times(-\rro,+\infty))$, and  is \emph{convex} and \emph{superlinear} in $v$,  see for instance \cite[Theorem A.2.6]{Cannarsa_Sinestrari_book}.
In what follows, we denote by \[L_\GG(x,v):=L(x,v,0)\] the Lagrangian associated to the critical  Hamiltonian $\GG(x, p)$.   
 This enables us to apply Aubry-Mather theory  to the critical Lagrangian $L_\GG$ to give  a characterization of the limit  $u^0$.
Even though the classical Aubry-Mather theory or weak KAM theory is established for $C^2$ Tonelli systems, most of the notions and results have their extensions to non-smooth ones, see Section \ref{section_pre} below. 

As we will see in Lemma \ref{L1_L4} below, the partial derivative of $L$ with respect to $u$ at $(x,v,0)$  exists and $L_u(x,v,0)\in C(TM)$. Let $h(y,x)$ be the Peierls barrier (see \eqref{def_peierls} for the definition ) and let $\wtM(L_\GG)$ be the set of all Mather measures (see \eqref{def_MaM} for the definition) for the Lagrangian $L_\GG$, we then address our second main result:
\begin{The}\label{mainresult2}
	Let $H(x,p,u)$ satisfy {\bf(H1)}, {\bf(H2*)}, {\bf(H3)} and {\bf(H4)}. Then
the limit solution $u^0$ obtained in Theorem \ref{mainresult1} can be characterized in either of the following two ways:
\begin{enumerate}[\rm(1)]
	\item $u^0(x)=\sup\limits_{w} w(x),$
	where the supremum is taken over all viscosity subsolutions $w$ of \eqref{eq_G} satisfying
\begin{equation*}
	\int_{TM} L_u(x,v,0)w(x)\,d\tilde{\mu}(x,v)\geqslant 0,	 \quad \text{for every~}  \tilde{\mu}\in \wtM(L_\GG).
\end{equation*}
	\item The limit solution 
	\begin{equation}\label{In_Math_meas}
		u^0(x)=\inf_{\tilde{\mu}\in \wtM(L_\GG)}\frac{\int_{TM}L_u(y,v,0)h(y,x)d\tilde{\mu}(y,v)}{\int_{TM}L_u(y,v,0)d\tilde{\mu}(y,v)}.
	\end{equation}
\end{enumerate}
\end{The}

\begin{Rem}[Vanishing discount problem]
For the discounted Hamiltonian, $L_u(x,v,0)$ is identically constant. Then Mather measure $\tilde{\mu}$ in the representation formula \eqref{In_Math_meas} can be replaced by its corresponding projected measure on $M$ defined by \[\mu:=\pi_\#\tilde{\mu},\] where $\pi: TM\to M$ is the canonical projection, namely $\mu(A)=\tilde{\mu}(\pi^{-1}(A))$ for each Borel subset $A\subset M$. Then, our theorem above coincides exactly with that of \cite{DFIZ2016}. 
\end{Rem}

 If we assume that $H$ satisfies Tonelli's conditions, then the following Lipschitz property, also known as Mather's graph property, holds:  the support  $\textup{supp}\tilde{\mu}\subset TM$ of each Mather measure $\tilde{\mu}$ is a Lipschitz graph over the base manifold $M$, i.e., the restriction of $\pi$ to $\textup{supp}\tilde{\mu}$ is a
 bi-Lipschitz homeomorphism. So, the measure $\tilde{\mu}$ in the representation formula \eqref{In_Math_meas} can still be replaced by its projected measure  $\mu$.
 However, this graphic property may not hold for general non-smooth and (non-strictly) convex Hamiltonians.

\subsection{Notes on our result and method}\label{methods}

Through non-smooth analysis,
we will generalize the techniques introduced in \cite{WYZ2020, DFIZ2016} to prove our convergence results. Compared with previous works, our result has the following features:

\begin{enumerate}[\rm(1)]
	\item The aforementioned condition {\bf(D3)} and the restriction \eqref{old_rem} in condition {\bf(D1)}, which were used in \cite{CCIZ2019}, can be totally removed now.
	\item Even for the nonlinear discounted system $H(x,p,u)=f(x,u)+F(x,p)$ that was discussed in \cite{Gomes_Mitake_Tran2018,CCIZ2019}, our result contains a novelty: the previous restriction \eqref{loc_lip} in condition {\bf(C2)} can be removed, since  $f\in C^1(M\times\R)$ with $f_u>0$  is enough to guarantee our assumptions {\bf(H3)--\bf(H4)}, see also Remark \ref{rem_1}.	
	\item Different from \cite{WYZ2020}, our proof is not  based on the dynamical theory of contact systems. So, we do not require smoothness on  $T^*M$, and only the partial derivative  $H_u(x,p, 0)$ is needed. In addition,
	 to achieve the convergence result (Theorem \ref{mainresult1}),  the fiberwise superlinear growth assumption can be weaken to the coercivity assumption. Also, we do not require the Lipschitz condition $|H_u|\leq C$.
 \item 
 Instead of using projected Mather measures on $M$,
 our characterization formula \eqref{In_Math_meas}  uses  Mather measures on the tangent bundle $TM$. This is because formula \eqref{In_Math_meas} is a weighted integral, where the weight function $L_u(x,v,0)$ shall depend on $v$ in general.  But Mather's graph property may not be true for general non-smooth Hamiltonians.
\end{enumerate}

Next, we  compare our method with that of previous works.  Analogous to \cite{DFIZ2016}, we need the tools of Aubry-Mather theory for non-smooth systems, especially the notion of Mather measure. 
Different from the approach of \cite{CCIZ2019} taking advantage of the comparison principle, here we will use an approximation scheme to select suitable Mather measures (see Section \ref{section_convergence}).  This approximation scheme will play a central role in the convergence problem.

To be more precise, we first review the vanishing discount problem.  It is well known that each solution $u^\lam$ of the equation $\lam u(x)$ $+\GG(x, Du(x))$ $=c(\GG)$ has the following expression: 
$$
	u^\lam(x)=\inf_{\xi^\lam\in \textup{Lip},\, \xi^\lam(0)=x}\int_{-\infty}^0 e^{\lam s}\Big[L_\GG(\xi^\lam(s),\dot{\xi}^\lam(s))+c(\GG)\Big]\,ds,
$$
which is  an infimum of Laplace transforms, and the infimum can be attained by a Lipschitz
 minimizing curve $\xilax:(-\infty, 0]\to M$ with $\xilax(0)=x$. Along this minimizer, the authors of \cite{DFIZ2016}  defined a probability measure $\tilde{\mu}_x^\lam$ by using a Laplace type average:
\begin{equation}\label{discount_meas}
\int_{TM} f(y,v) d\tilde{\mu}^\lam_x(y,v):= \lam\int_{-\infty}^0e^{\lam s}    f(\xilax(s),\dxilax(s))\, ds, \quad\text{for every~} f\in C_c(TM).
\end{equation}
They have showed that all accumulation points of the sequence $\{\tilde{\mu}_x^\lam\}_{\lam>0}$ are Mather measures. The asymptotic convergence of solutions then follows from this approximation scheme. 

For general Hamiltonian $H(x, p, u)$ depending nonlinearly on $u$, the solution $u^\lam$ has no expression as an infimum of Laplace transforms. But if  $H(x, p, u)$ is smooth and satisfies some Tonelli-type conditions, by weak KAM theory of smooth contact Hamiltonians \cite{Wang_Wang_Yan2019_Variational} the solution $u^\lam$ of equation \eqref{eq_H_lambda} admits a $C^1$-minimizing curve $\xilax$ satisfying
\[u^\lam(\xilax(0))=u^\lam(\xilax(-T))+\int_{-T}^0L\Big(\xilax(s),\dxilax(s),\lam u^\lam(\xilax(s))\Big)+c(\GG)\,ds,\quad \text{for all~} T>0.\]
As a generalization of \eqref{discount_meas}, the authors of \cite{WYZ2020} defined a Radon measure $\tilde{\mu}_x^\lam$ by
\begin{align}\label{WYZ_meas}
	\int_{TM} f\, d\tilde{\mu}^\lam_x
	:= \lam \int_{-\infty}^0    f(\xilax(s),\dxilax(s)) \exp\left[-\lam\int_0^s\int_0^1 L_u(\xilax(t),\dxilax(t),\tau \lam u^\lam(\xilax(t))) d\tau dt\right] ds
\end{align} 
for every $f\in C_c(TM)$.  

The above definition of measure cannot be applied directly to the setting of the current paper, since our assumption {\bf(H4)} only assumes that   $H_u(x,p, u)$ exists  at $u=0$, which therefore means $L_u(x, v, u)$ exists only at $u=0$ (see Lemma \ref{L1_L4}). 
To overcome this difficulty,  we introduce a new Lagrangian $\LL^\lam(x, v):=L(x,v,\lam u^\lam(x))$ on $TM$, for each solution $u^\lam$ of equation \eqref{eq_H_lambda}. Then, it is easy to observe that $u^\lam$ is still a weak KAM solution of $\LL^\lam$, which therefore gives rise to a Lipschitz minimizing curve $\xilax (s):(-\infty, 0]\to M$ with respect to $u^\lam$ and $\LL^\lam$. Now, inspired by the treatment of \cite{WYZ2020}  on proving \eqref{WYZ_meas} converging to a Mather measure, 
 we define a  probability measure  $\tilde{\mu}^\lam_x$ as follows:
 \begin{equation}\label{our_meas}
	\int_{TM} f(y,v) d\tilde{\mu}^\lam_x(y,v):= \frac{\int_{-\infty}^0    f(\xilax(s),\dxilax(s)) e^{-\lam\int_0^s L_u(\xilax(\tau),\dxilax(\tau),0)\ d\tau}\, ds}{ \int_{-\infty}^0   e^{-\lam\int_0^s L_u(\xilax(\tau),\dxilax(\tau),0)\ d\tau}\, ds},
\end{equation}
where  only $L_u(x,v,0)$ is needed in the exponent term. 
 We can still prove that the accumulation points of the sequence $\{\tilde{\mu}_x^\lam\}_{\lam>0}$ are Mather measures. See Section \ref{section_convergence} for more details.

\medskip 

The paper is organized as follows. Section \ref{section_pre} recalls some  basic facts on viscosity solutions of the Hamilton-Jacobi equations, and also collects some necessary results in
 Aubry-Mather theory and weak KAM theory under a non-smooth setting. In Section \ref{section_exist_and_uniq}, under assumptions {\bf(H2)--(H3)}, we show the existence and uniqueness of  viscosity solutions  to  equation \eqref{eq_H_lambda} for each small parameter $\lam>0$. We also give some key estimates for the whole family of solutions. Then, as a corollary, we study a special case and obtain a convergence result under the assumption that constant functions are critical subsolutions.
Section \ref{section_convergence} is the main part of this paper, where we prove our convergence result under assumptions {\bf(H1)--(H4)}. To prove Theorem \ref{mainresult1}, we first provide a possible limit $u^0$, then we show that  every accumulation point $u$ of the family of solutions $\{u^\lam\}_{\lam>0}$ is equal to $u^0$. To this end, we need an approximation scheme to select suitable Mather measures. To prove Theorem \ref{mainresult2},  we make use of Lemma \ref{new_gen_2} and the properties of Peierls barrier and Mather measures.

\section{Preliminaries}\label{section_pre}

In this section, we first recall some basic facts about the viscosity solutions of the Hamilton-Jacobi equations. Next,
we provide  some useful results from Aubry-Mather theory and weak KAM  theory which are necessary for the purpose of this paper. The classical 
Aubry-Mather theory  is  established for $C^2$ Tonelli systems, see Mather's original papers \cite{Mather1991, Mather1993}.  As for weak KAM theory of $C^2$ Tonelli systems, the reader can refer to Fathi's book \cite{Fathi_book}. Here, we also refer the interested reader to \cite{Maro_Sorrentino2017, Wang_Wang_Yan2019_Variational, Wang_Wang_Yan2019_Aubry, CCWY2019, Cannarsa2019herglotz} 
for an analogue of Aubry-Mather theory or weak KAM theory of contact Hamiltonians.

However,  our systems in this paper are required to be purely continuous on $T^*M$ and are therefore lack of Hamiltonian or Lagrangian dynamics. Thus  the generalizations of Aubry-Mather theory and weak KAM theory for non-smooth systems are needed. The main references are  \cite{Fathi_Siconolfi2005, Davini_Zavidovique2013,DFIZ2016,Ishii2008}. In what follows, $M$ is always assumed to be a connected and compact smooth manifold without boundary.

\subsection{Viscosity solutions of the Hamilton-Jacobi equations}
Let us first review the concept of viscosity solutions. We consider the following stationary Hamilton-Jacobi equation
\begin{equation}\label{def_hj_eq}
	F(x, Du(x), u(x))=c,\quad x\in M,
\end{equation}
where $F\in C(T^*M\times\R)$.
	 A function $u: M\to\R$ is a \emph{viscosity subsolution} of the H-J equation \eqref{def_hj_eq} if, for any  $\varphi\in C^1(M)$ and any $x_0\in M$ satisfying $\varphi \geqslant u$  and $\varphi(x_0)=u(x_0)$, we have 
	 \[F(x_0, D\varphi(x_0), \varphi(x_0))\leqslant c.\]
Similarly, $u$ is a \emph{viscosity supersolution} if, for any  $\psi\in C^1(M)$ and any $x_0\in M$ satisfying  $\psi\leqslant u$  and $\psi(x_0)=u(x_0)$, we have 
\[F(x_0, D\psi(x_0), \psi(x_0)) \geqslant c.\] 
Finally,  $u$ is a \emph{viscosity solution} of \eqref{def_hj_eq} if it is both a subsolution and a supersolution.

For the applications of the current paper, we need  the following comparison principle. 
\begin{Lem}\label{The_comp_prin}
 Let $F(x,p,u)\in C(T^*M\times\R)$ be strictly increasing in $u$. Assume that the equation 
 \begin{equation*}
   	F(x ,Du(x), u(x))=c, \quad x\in M ,
   \end{equation*}
  admits a Lipschitz continuous viscosity solution. Then for any  continuous viscosity subsolution $f$ and any continuous  viscosity supersolution $g$, we have   	
 $f \leqslant g $.
   \end{Lem}
   This result is well known in the literature. 
  For the case where $M$ is a compact manifold, we refer the reader to  \cite[Theorem 3.2]{CCIZ2019} for a complete proof.

Next, we focus on the Hamiltonians defined on the cotangent bundle. Let $H(x,p)$ be continuous on $T^*M$, then we state two assumptions that we may use in the sequel
\begin{enumerate}
	\item (Convexity) $H$ is \emph{convex} in  $p$; 
	\item (Coercivity) $H$ is \emph{coercive} in $p$, i.e., 
$\lim_{|p|_x\to+\infty}H(x,p)=+\infty$, uniformly in  $x\in M.$
\end{enumerate}

 The following result is classical,  see for instance \cite{Barles_book}.
\begin{Pro}\label{sublip}
Let $H(x,p)\in C(T^*M)$ be coercive in $p$ and let $c\in\R$. Then
	any continuous viscosity subsolution of $H(x, Du(x))=c$ is  Lipschitz continuous. Moreover, the Lipschitz constant is bounded above by $\rho_c$ with
	$$\rho_c=\sup\{|p|_x : H(x,p)\leq c\}.$$
\end{Pro}

Now, we define the {\em critical value} $c(H)$ associated with $H$ as follows: 
\begin{equation}
	c(H):=\inf\{c\in\R ~:~ H(x,Du(x))=c\ \text{~admits a viscosity subsolution}\}. 
\end{equation}
Then we have the following well known property, see \cite{Lions_Papanicolaou_Varadhan1987}. 
\begin{Pro}\label{crit_val_solu}
	Let $H(x,p)\in C(T^*M)$ be coercive in $p$. Then the critical value $c(H)$ is finite, and $c(H)$ is the unique real number $c$  such that the H-J equation
$H(x,Du(x))=c,$ $x\in M,$
admits viscosity solutions.
\end{Pro}

As $M$ is compact, we use $\|f\|_\infty:=\sup_{x\in M}|f(x)|$ to denote the sup-norm  of the  function $f\in C(M)$.
For our applications, we need  a standard smooth approximation result  for subsolutions.  See for instance \cite[Theorem 10.6]{Fathi2012} for the case where $M$ is a compact manifold.

\begin{Pro}\label{smooth_sub_approx}
Let $H(x,p)\in C(T^*M)$ be  convex in $p$.
If $u: M\to\R$ is a  Lipschitz viscosity subsolution of the equation $H(x, Du(x))$ $=c$, then for every $\varepsilon>0$, we can find a $C^\infty$ function $w: M\to\R$	such that $\|u-w\|_\infty\leqslant \varepsilon$, $H(x, Dw(x))\leqslant c+\varepsilon$ for every $x\in M$, and $\|Dw\|_\infty\leqslant\textup{Lip}(u)+1$. 
\end{Pro}

\subsection{Aubry-Mather theory for non-smooth systems}

Throughout this subsection, $H: T^*M\to \R $ is assumed to be a \emph{continuous} Hamiltonian which is \emph{convex} in $p$, and is fiberwise \emph{superlinear}, i.e.,
\[\lim_{|p|_x\to\infty} H(x,p)/|p|_x=+\infty, \quad  \text{uniformly in~}  x\in M.\]
It is quite clear that the superlinearity implies the coercivity.
 Through the Legendre-Fenchel transformation, we can define the Lagrangian $L:TM\to\R $ associated to $H$ as follows:
 \begin{align*}
	L(x,v)=\sup_{p\in T^*_xM}\{\du{p,v}_{x}-H(x,p)\},\quad  \forall~(x,v)\in TM,
\end{align*}
where $\du{p,v}_{x}$ denotes the value of the linear form $p\in T_x^*M$ evaluated at $v\in T_xM$. 
It is a classical result in convex analysis that  $L(x, v)$ is \emph{continuous} on $TM$, and is also \emph{convex} and \emph{superlinear} in the fibers.

For every $t>0$ and $x, y\in M$, let $\Gamma^t_{x,y}$ denote the set of Lipschitz continuous curves $\xi: [0,t]\to M$ with $\xi(0)=x$ and $\xi(t)=y$.  We define the {\em action function}
\begin{align*}
	h_t(x,y)=\inf_{\xi\in \Gamma^t_{x,y}}\int_{0}^{t}[L(\xi(s),\dot{\xi}(s))+c(H)]\ ds.
\end{align*}
The infimum can be attained by a Lipschitz continuous minimizing curve.
Then, we define a real-valued function $h$ on $M\times M$ by  
\begin{align}\label{def_peierls}
	h(x,y):=\liminf_{t\to+\infty} h_t(x,y).
\end{align}
In the literature, $h(x,y)$ is called the {\em Peierls barrier}. 
This leads us to define the so-called {\em projected Aubry set} $\mathcal{A}$  by
\begin{align*}
	\mathcal{A}:=\{x\in M: h(x, x)=0\}.
\end{align*}
Since $M$ is compact, the projected Aubry set $\mathcal{A}$ is non-empty and compact.

\begin{Pro}\label{proper_criti_sol}
The following properties hold:
\begin{enumerate}[\rm(1)]
	\item $h$ is finite valued and Lipschitz continuous.
     \item For any $y\in M$, the function $x\mapsto h(y,x)$  is a viscosity solution of  
$H(x,Du(x))=c(H)$ and the function $x\mapsto -h(x,y)$  is a viscosity subsolution of  $H(x,Du(x))=c(H)$.
	\item If $v$ is a  viscosity subsolution, then $v(y)-v(x)\leqslant h(x,y)$.
	\item If $f$ and $g$ are a viscosity subsolution and a viscosity supersolution 
     of $H(x, Du(x))=c(H)$, and $f\leq g$ on the set $\mathcal{A}$, then $f\leq g$.
     \item If $v$ is a viscosity solution, then $v(x)=\min_{y\in M}[v(y)+h_t(y,x)]$  for every $t>0$.
\end{enumerate}
\end{Pro}
These results are well known for $C^2$ Tonelli systems \cite{Fathi_book}.  For the non-smooth case, see \cite{DFIZ2016} for items (1)-(4). As for item (5), see for instance \cite[Proposition 6.5]{CCIZ2019} for the case where $M$ is a compact manifold. Indeed,
we can define  a function $U(x, t)= \min_{y\in M}[v(y)+h_t(y,x)]$. It is a viscosity solution of the Cauchy problem: $\partial_t u(x,t)+H(x, \partial_x u(x,t))=c(H)$ and  $u(x,0)=v(x)$. Note that $v(x)$ is also a viscosity solution of this Cauchy problem. Then the uniqueness of solutions leads to $v(x)=U(x,t)$.

The following semi-continuity property is well known, see for instance \cite{Buttazzo_Giaquinta_Hildebrandt_book}. 
\begin{Lem}\label{lower_semi_cont}
 Let $\xi_k:[a,b]\to M$	be a sequence of equi-Lipschitz curves such that
$$\sup_k\int_{a}^{b} L(\xi_k(s),\dot{\xi}_k(s))\,ds\leqslant C<\infty,$$
then there is a subsequence $\{\xi_{k_i}\}_i$  converging uniformly to a Lipschitz curve $\xi:[a,b]\to M$, and 
\begin{equation}\label{lower_con}
	\int_{a}^{b} L(\xi(s),\dot{\xi}(s))\,ds\leqslant\liminf_{i\to\infty}\int_{a}^{b} L(\xi_{k_i}(s),\dot{\xi}_{k_i}(s))\,ds.
\end{equation}
\end{Lem}

Next, we will  introduce the notion of  Mather measure. For the purpose of this paper, it is convenient to adopt the equivalent definition originating from Ma\~n\'e \cite{Mane1996}.
Recall that a Borel probability measure $\tilde{\mu}$ on $TM$ is called {\em closed} if it satisfies $\int_{TM} |v|_x\ d\tilde{\mu}<+\infty$ and
 $$\int_{TM} \du{D\varphi,v}_x\ d\tilde{\mu}(x,v)=0,\quad\text{ for all~} \varphi\in C^1(M).$$ 
 Let $\mathcal{P}$ be the set of closed probability measures on $TM$.  
This set is non-empty. In addition, the critical value $c(H)$ can be obtained by considering a minimizing problem on the set $\mathcal{P}$. More precisely, we have :
\begin{Pro}The following relation holds: 
	\begin{align}\label{mini_pro}
	-c(H)=\min_{\tilde{\mu}\in\mathcal{P}}\int_{TM}L(x,v)\ d\tilde{\mu}(x,v).
\end{align}
\end{Pro}

A {\em Mather measure} for the Lagrangian $L$ is a  measure $\tilde{\mu}\in\mathcal{P}$  achieving the minimum of \eqref{mini_pro}:
\begin{align}\label{def_MaM}
	\int_{TM}L(x,v)\ d\tilde{\mu}=-c(H).
\end{align}
 We denote by 
 $\widetilde{\mathfrak{M}}(L)$ the set of all Mather measures. This set is non-empty and convex. Note that our Proposition \ref{weak_convergence} provides an approximation scheme to obtain Mather measures.

We end up this section by the following compactness property:
\begin{Pro}\label{MM_compactness}
	Each Mather measure $\tilde{\mu}$ has compact support, i.e., the set $\text{supp}\tilde{\mu}$ is compact in $TM$.
\end{Pro}

This compactness result is well known  for $C^2$ Tonelli systems  \cite{Mather1991}. For the case where the Hamiltonian is purely continuous and is convex and superlinear in the fibers, a complete proof can be found in the  Appendix B (Theorem B.22) of the arXiv version of the work \cite{DFIZ2016}.

\section{Some \emph{a priori} estimates for solutions}\label{section_exist_and_uniq}
In this section,  we will show the existence and uniqueness of  solutions to  equation \eqref{eq_H_lambda},  and give key estimates for the whole family of solutions. This is a standard result  by using  Perron's method (see for instance \cite{Ishii1987}) and the comparison principle.  It is worth noting that assumptions {\bf(H1)} and {\bf(H4)} are not needed here.

\begin{Pro}\label{uni_bound}
Let $H$ satisfy  {\bf(H2)} and {\bf(H3)}. There exists a  number $\alo>0$ such that for each $\lambda\in(0,\alo)$, equation \eqref{eq_H_lambda} admits a unique continuous viscosity solution $u^\lambda$.  In addition,  the  family $\{u^\lambda\}_ {\lambda\in(0,\alo)}$ is  equi-bounded and equi-Lipschitz, namely there exist constants $\mathbf{C_0}>0$ and $\mathbf{M_0}>0$ which are independent of $\lambda$ such that
\begin{equation}\label{apriori_est}
	\|u^\lambda\|_\infty\leqslant \mathbf{C_0},\quad \textup{Lip}(u^\lambda)\leqslant \mathbf{M_0}.
\end{equation}
\end{Pro}

\begin{proof}
We will construct solutions by Perron's method.  First, we recall that  $\GG(x,p)$ is coercive in the fibers as a consequence of  assumptions {\bf(H2)-(H3)}, see also \eqref{coercive_forall_u}. According to Proposition \ref{crit_val_solu} and Proposition \ref{sublip}, equation \eqref{eq_G}  admits  a Lipschitz viscosity solution, denoted by $w_0$. Setting $w_0^-:=w_0-\|w_0\|_\infty$  and $w_0^+:=w_0+\|w_0\|_\infty$,  the monotonicity assumption {\bf(H3)} implies that $w_0^-$ and  $w_0^+$ are, respectively, a Lipschitz subsolution and a Lipschitz  supersolution of \eqref{eq_H_lambda}. Now we fix two constants that are independent of $\lam$:
\begin{equation}\label{fsgw}
	\mathbf{C_0}:=2\|w_0\|_\infty,\quad \alo=\rro/\mathbf{C_0},
\end{equation}
where $\rro$ is given in assumption {\bf(H2)}.
For each $\lambda\in (0,\alo)$, we set  
\begin{equation}\label{perron_construction}
u^\lam(x):=\sup\{v(x):w_0^-\leqslant v\leqslant w_0^+, v \text{~is a continuous subsolution of~} \eqref{eq_H_lambda} \}.
\end{equation}
Of course, $u^\lam$ is well-defined as $w_0^-$ itself is an admissible subsolution in the above formula.
Observe that if $v$ is a subsolution of \eqref{eq_H_lambda} satisfying  $w_0^- \leqslant v \leqslant w_0^+$, then 
\begin{equation}\label{norm_est1}
	|v(x)|\leqslant \mathbf{C_0},\quad \text{for all~} x \in M.
\end{equation}
This gives $\|\lam v\|_\infty< \rro$ and $H(x, p,-\rro)\leq H(x, p,\lam v(x))$, and 
 $v$ is then a viscosity subsolution of the following equation 
$$ H(x, Dv(x),-\rro) = c(\GG).$$
This indicates that $v$ is Lipschitz continuous thanks to Proposition \ref{sublip}.
Using the coercivity assumption {\bf(H2)}, we can find a constant 
$\mathbf{M_0} >0$, such that 
 $H(x, p,-\rro)> c(\GG)$ for all $|p|_x>\mathbf{M_0}$.
This in turn gives  a uniform Lipschitz bound $\mathbf{M_0}$ for the function $v(x)$.  
As a result, formula \eqref{perron_construction} can be rewritten as 
\begin{equation}
u^\lambda(x)=\sup\{v(x):w_0^-\leqslant v\leqslant w_0^+, v \text{~is a Lipschitz subsolution of~} \eqref{eq_H_lambda}, ~\text{Lip}(v)\leq \mathbf{M_0} \}.
\end{equation}
By the Perron method, $u^\lam$ is exactly a viscosity solution of \eqref{eq_H_lambda}, and also satisfies the following Lipschitz property,
\begin{equation}\label{norm_est2}
|u^\lam(x)-u^\lam(y)|\leqslant \mathbf{M_0}|x-y|,\quad \text{for every~} x, y\in M. 
\end{equation}

As for the uniqueness of solutions, it follows directly from the comparison principle (see Lemma \ref{The_comp_prin}).  Finally, the uniform estimate \eqref{apriori_est}	is a consequence of inequalities \eqref{norm_est1} and \eqref{norm_est2}.
\end{proof}

As a corollary, we can consider a special model which satisfies that $\GG(x,0)\leq c(\GG)$ for all $x\in M$. In this case, we can easily obtain a convergence result without the convexity assumption {\bf(H1)}. 
\begin{Pro}\label{uni_const}
Let $H$  satisfy  assumptions {\bf(H2)}-{\bf(H3)}. Suppose that constant functions are subsolutions of equation \eqref{eq_G}, or equivalently $\GG(x, 0)\leq c(\GG)$ for all $x\in M$. Then the  solution $u^\lam\geq 0$ with $\lam\in(0,\alpha_0)$,  and converges uniformly  to a critical solution $u^0$ as $\lambda\to 0$.
\end{Pro}
\begin{proof}
Let $\lam\in(0,\alpha_0)$.
	By Proposition \ref{uni_bound} we have already known that the sequence of solutions $u^\lam$ of equation  \eqref{eq_H_lambda} is equi-bounded and equi-Lipschitz. Since constants are viscosity subsolutions of equation \eqref{eq_G}, it follows that the function $w(x)\equiv 0$ is a viscosity subsolution of equation \eqref{eq_H_lambda}. So, $u^\lam\geq w=0$ as a consequence of the comparison principle (see Lemma \ref{The_comp_prin}). Then, for $\lam'<\lam$, we have
	\[H(x, Du^{\lam}(x), \lam'\,u^{\lam}(x) )\leq H(x, Du^{\lam}(x), \lam \,u^{\lam}(x) )=c(\GG).\]
in the viscosity sense, and thus $u^\lam$ is a viscosity subsolution of the equation $H(x, Du(x), \lam' u(x))$ $=c(\GG)$. By the comparison principle again we get that $u^\lam\leq u^{\lam'}$. Therefore, the sequence $\{u^\lam\}_{\lam\in (0,\alpha_0)}$ is monotone. The convergence is now evident from what we have proved.	
\end{proof}

\section{Convergence of the family of solutions}\label{section_convergence}

In this section we will prove our main results Theorem \ref{mainresult1} and Theorem \ref{mainresult2}. 
Since we are only interested in the asymptotic behavior of the family $\{u^\lam\}_\lam$ as $\lam\to 0$, we can always let $\lam\in(0,\alo)$ throughout this section, where $\alo$ is provided in Proposition \ref{uni_bound}.
As is seen later, without affecting our analysis, we can always reduce to the case where the Hamiltonian satisfies the superlinearity assumption {\bf(H2*)}. 

To explain this, we recall that  Proposition \ref{uni_bound} shows that  the family of solutions $\{u^\lam\}_{\lam\in (0,\alo)}$ is equi-Lipschitz continuous, that is, $\textup{Lip}(u^\lambda)\leqslant \mathbf{M_0}.$
Thus it will not affect our study of convergence if we modify $H$ outside the compact set $\{(x,p)\in T^*M~:~ |p|_x\leq  \mathbf{M_0}\}$. For instance, let $f(x,p)$ be a continuous Hamiltonian which is convex and superlinear in $p$, and $f\equiv 0$ on the compact set $\{(x,p)\in T^*M~:~ |p|_x\leq  \mathbf{M_0}\}$. Then we consider a modified Hamiltonian $$\widetilde{H}(x, p, u)=H(x, p, u)+f(x, p).$$
It is clear that  $\widetilde{H}$ satisfies assumptions {\bf(H1)}, {\bf(H2*)}, {\bf(H3)} and {\bf(H4)}. Since $\widetilde{H}(x, p, u)=H(x, p, u)$ on the set  $\{(x, p, u)\in T^*M\times\R~:~ |p|_x\leq  \mathbf{M_0}\}$,  the function $u^\lam$ is still a viscosity solution of the following Hamilton-Jacobi equation:
\begin{equation}\label{ModHJ}\tag{MHJ$_\lambda$}
	\widetilde{H}(x, Du(x), \lam u(x))=c(\GG).
\end{equation}
In particular, the modified critical Hamiltonian $\widetilde{\GG}(x,p)=\GG(x,p)+f(x, p)$ has the same critical value with the unmodified one, namely $c(\widetilde{\GG})=c(\GG)$. Therefore, to prove Theorem \ref{mainresult1}, it is equivalent to  
study the convergence of solutions to the H-J equations \eqref{ModHJ}.

In the remainder of this paper, by what we have shown above \textbf{we will always assume that $H(x, p, u)$ satisfies the superlinearity assumption {\bf(H2*)}} instead of the coercivity assumption {\bf(H2)}. This enables us to introduce the associated Lagrangian $L(x, v, u)$ defined in \eqref{def_Lag}. 

Recall that for every  fixed $u_0\in(-\rro,+\infty)$  the map $p\mapsto H(x, p, u_0)$ is superlinear.
So, $L(x,v,u)$ is finite-valued whenever $u\in(-\rro,+\infty)$.  But $L(x,v,u)$ may take the value $+\infty$ when $u<-\rro$. However, as we will see later, since we are only interested in the asymptotic behavior of the family of solutions, only the information on $TM\times(-\rro,+\infty)$ is needed for our purpose.

\subsection{Some lemmas}
Let us stress that, from now on, we only discuss the Lagrangian restricted on the region $TM\times(-\rro,+\infty)$, i.e. \[L(x,v,u): TM\times(-\rro,+\infty)\to\R. \]

It is a standard result that $L(x,v,u)\in C(TM\times(-\rro,+\infty))$, and is \emph{convex} and \emph{superlinear} in $v$. 
For every $(x,u)\in M\times(-\rro,+\infty)$, we denote by $D_2^-L(x,v,u)$ the  (Fr\'echet)  subdifferential of the convex function $ L(x,\cdot, u)$ at $v$, which is the set of all $p$ such that
$L(x,v',u)\geq L(x,v,u)+\du{p, v'-v}_x$  for all $v'\in T_xM.$
Similarly, we denote by
$D_2^-H(x,p,u)$ the  (Fr\'echet)  subdifferential of the convex function $ H(x,\cdot, u)$ at $p$.
The following results are classical in convex analysis: for $u> -\rro$,

\begin{enumerate}[(i)]
	\item  $D_2^-L(x,v,u)$ and $D_2^-H(x,p,u)$ are both non-empty, compact and convex sets.
	\item The following equivalence relations hold: 
\begin{equation}\label{superdiff_equiv}
	v\in D_2^-H(x,p,u)\Longleftrightarrow p\in D_2^-L(x,v,u)\Longleftrightarrow L(x,v,u)=\du{p, v}_x-H(x,p,u).
	 \end{equation}	
    \item For every compact subset $X\subset TM\times(-\rro,+\infty)$ there exists  a constant $K>0$ such that $|p|_x\leqslant K$, for all $p\in D_2^-L(x,v,u)$ and all $(x,v,u)\in X$.  This is a consequence of the superlinearity and \eqref{superdiff_equiv}. Similar result holds for $D_2^-H(x,p,u)$.
    \item  The set-valued maps $(x,p,u)\mapsto D_2^-H(x,p,u)$ and 
$(x,v,u)\mapsto D_2^-L(x,v,u)$ are both upper semi-continuous.
\end{enumerate}
We refer the reader to \cite[Appendix A]{Cannarsa_Sinestrari_book} for more details. With these properties, we can prove the following lemma which will be useful later.

\begin{Lem}\label{L1_L4}
	Let $H(x,p,u)$ satisfy {\bf(H1)}, {\bf(H2*)}, {\bf(H3)} and {\bf(H4)}. Then the associated  Lagrangian $L(x,v,u)$, restricted on $TM\times(-\rro,+\infty)$, satisfies the following properties: 
	\begin{enumerate}[(1)]
	\item  $L\in C(TM\times(-\rro,+\infty))$  is convex and superlinear in the fibers, and is strictly decreasing in $u$.  In particular, for each $u_0\in (-\rro,+\infty)$,
\begin{equation*}\label{about_L_super}
	\lim_{|v|_x\to+\infty}\frac{L(x, v, u_0)}{|v|_x}=+\infty,\quad \textup{uniformly in~} x\in M.
\end{equation*}
		\item the partial derivative of $L$ with respect to $u$ at the point $(x,v,0)$  exists and $L_u(x,v,0)<0$.
		\item    $L_u(x,v,0)=-H_u(x,\wtp,0)$, for all $\wtp\in D^-_2L(x,v,0)$. In particular, $L_u(x, v, 0)\in$ $C(TM)$.
		\item the following property holds:
		 \begin{align}\label{ap_local_uni}
	\lim_{u\to 0}\frac{L(x,v,u)-L(x,v,0)}{u}=L_u(x,v,0),\quad \text{locally uniformly in~}  (x, v)\in TM.
\end{align}  
	\end{enumerate}
\end{Lem}
\begin{proof}
(1)~ This is standard in convex analysis, see for instance \cite[Theorem A. 2.6]{Cannarsa_Sinestrari_book}.

(2)~ 
 We first show that the partial derivative $L_u(x, v, 0)$ exists. 
Let us fix a point $(x, v)\in TM$. For each $r\in\R\setminus\{0\}$ and small $\vep>0$, by definition,  for all $\wtp\in D_2^-L(x,v,0)$ we have
\begin{align*}
	L(x,v,0)=\du{\wtp, v}_x-H(x,\wtp,0),\quad
	L(x,v,\vep r)\geq \du{\wtp, v}_x-H(x,\wtp,\vep r).\end{align*}
 Hence 
	$L(x,v,\vep r)-L(x,v,0)\geq H(x,\wtp, 0)-H(x,\wtp, \vep r).$
Dividing both sides  by $\vep>0$ and sending $\vep\to 0$, one gets
\begin{equation}\label{left_ine1}
	\liminf_{\vep\to 0^+}\frac{L(x,v,\vep r)-L(x,v,0)}{\vep}\geq-H_u(x,\wtp,0)r,\quad \text{~for all~} \wtp\in D_2^-L(x,v,0).
\end{equation}
Next, we consider the  supremum limit,  and pick a subsequence $\vep_n\to 0^+$  satisfying 
\begin{equation}\label{left_ine2}
	\lim_{n\to \infty}\frac{L(x,v,\vep_n r)-L(x,v,0)}{\vep_n}=\limsup_{\vep\to 0^+}\frac{L(x,v,\vep r)-L(x,v,0)}{\vep}.
\end{equation}
For each $n$, we pick one element $p_n\in D^-_2L(x,v,\vep_n r)$. The sequence $\{p_n\}_n$ is bounded, then, without loss of generality, we can assume $p_n$ converges to some point $p_0$. Note that $p_0\in D^-_2L(x,v,0)$  as a result of the upper semi-continuity. By definition, 
\begin{align*}
	L(x,v,0)\geq \du{p_n, v}_x-H(x,p_n,0),\quad
	L(x,v,\vep_n r)= \du{p_n, v}_x-H(x,p_n,\vep_n r),
\end{align*}
which leads to
\begin{equation}\label{ineq_vepn_1}
	\frac{L(x,v,\vep_n r)-L(x,v,0)}{\vep_n}\leq \frac{H(x,p_n, 0)-H(x,p_n, \vep_n r)}{\vep_n}.
\end{equation}
Since $\{p_n\}_n$ is bounded and $M$ is compact, it follows from the locally uniform convergence \eqref{H_locaunif_conv} of $H$ that for each $\delta>0$, there exists $N=N(\delta)$ such that 
\begin{equation*}
	\left|\frac{H(x,p_n, \vep_n r)-H(x,p_n, 0)}{\vep_n r} -H_u(x, p_n, 0)\right|<\delta, \quad\text{for all} ~ n\geq N.
\end{equation*}
Hence, \eqref{ineq_vepn_1} is now reduced to
\begin{equation*}
\frac{L(x,v,\vep_n r)-L(x,v,0)}{\vep_n}\leq -H_u(x, p_n, 0)r+r\delta, \quad\text{for all} ~ n\geq N.
\end{equation*}
As we have assumed $H_u(x,p,0)\in C(T^*M)$ in assumption {\bf(H4)}, 
\begin{equation*}
\lim_{n\to\infty}\frac{L(x,v,\vep_n r)-L(x,v,0)}{\vep_n}\leq -H_u(x, p_0, 0)r+r\delta.
\end{equation*}
Sending $\delta\to 0$ and combining with inequality \eqref{left_ine2},
\begin{equation}\label{left_ine3}
\limsup_{\vep\to 0^+}\frac{L(x,v,\vep r)-L(x,v,0)}{\vep}\leq -H_u(x, p_0, 0)r.
\end{equation}
Therefore, by the arbitrariness of $r\in\R\setminus\{0\}$, we conclude from inequalities \eqref{left_ine1} and \eqref{left_ine3} that 
\begin{equation*}
\lim_{u\to 0}\frac{L(x,v,u )-L(x,v,0)}{u}= -H_u(x, p_0, 0).
\end{equation*}
This finishes the proof of item (2).

(3)~ Observe that inequalities \eqref{left_ine1} and \eqref{left_ine3} also imply that
$-H_u(x,p_0,0) r\geq -H_u(x,\wtp,0)r$ for all $~\wtp\in D_2^-L(x,v,0)$ and all $r\in\R\setminus\{0\}$. This means 
$H_u(x,p_0,0)=H_u(x,\wtp,0)$ for all $\wtp\in D_2^-L(x,v,0)$.
We therefore obtain
\begin{equation}\label{dsn_choice_of_p}
	L_u(x,v,0)=-H_u(x,\wtp,0), \quad\text{for all}~\wtp\in D_2^-L(x,v,0).
\end{equation}
Thanks to $H_u(x,p,0)$ $\in C(T^*M)$, $L_u(x, v,0)\in C(TM)$ follows from  \eqref{dsn_choice_of_p} and the upper semi-continuity of the set-valued map $(x, v)\mapsto D^-_2L(x,v,0)$.

(4)~ Now, it only remains to prove the locally uniform convergence \eqref{ap_local_uni}. Let $B\subset TM$ be a compact subset, it suffices to prove that for every $\vep>0$, there exists $\delta\in(0,\rro)$ such that 
\begin{align}\label{prf_L_luc}
	\left|\frac{L(x,v,u)-L(x,v,0)}{u}-L_u(x,v,0)\right|< \vep,\quad \text{for all~}  (x, v)\in B, 0<|u|< \delta.
\end{align}
In fact, for each $(x,v)\in B$, we pick one element $\px\in D_2^-L(x,v,0)$ and one element $\pxu\in D_2^-L(x,v,u)$. This implies
\begin{align*}
	L(x,v,0)=\du{\px, v}_x-H(x,\px,0),\quad
	L(x,v,u)= \du{\pxu, v}_x-H(x,\pxu,u),\\
	L(x,v,u)\geq \du{\px, v}_x-H(x,\px,u),\quad L(x,v,0)\geq \du{\pxu, v}_x-H(x,\pxu,0).
\end{align*}
Then we derive 
\begin{equation*}
	H(x,\px,0)-H(x,\px, u)\leq L(x,v,u)-L(x,v,0)\leq H(x,\pxu,0)-H(x,\pxu, u).
\end{equation*}
 Using $L_u(x,v,0)=-H_u(x,\px,0)$ one gets that for $u\neq 0$,
 \begin{align}\label{maxI1I2}
 	\left|\frac{L(x,v,u)-L(x,v,0)}{u}-L_u(x,v,0)\right|\leq\max\{I_1(x,v,u), I_2(x,v,u)\}
 \end{align}
 where 
 \begin{equation*}
 	I_1(x,v,u):=\left|\frac{H(x,\px,u)-H(x,\px,0)}{u}-H_u(x,\px,0)\right|
 \end{equation*}
and
\begin{equation*}
	I_2(x,v,u):=\left|\frac{H(x,\pxu,u)-H(x,\pxu,0)}{u}-H_u(x,\px,0)\right|
\end{equation*}
For each $\vep>0$, by the locally uniform convergence \eqref{H_locaunif_conv} of $H$, there exists $\delta_1\in(0,\rro)$ such that 
\begin{equation}\label{I11}
	|I_1(x,v,u)|<\vep,\quad\text{for all~} (x,v)\in B,~ 0<|u|< \delta_1.
\end{equation}
Here, we have used the fact that  ${\px}$ is  bounded uniformlly for all $(x,v)\in B$. Similarly, the locally uniform convergence \eqref{H_locaunif_conv}  also implies that  there exists $\delta_2\in(0,\rro)$ such that 
\begin{equation}
	\left|\frac{H(x,\pxu,u)-H(x,\pxu,0)}{u}-H_u(x,\pxu,0)\right|<\frac{\vep}{2},~\text{for all~} (x,v)\in B, ~0<|u|< \delta_2.
\end{equation}
Here,  we have used the fact that  ${\pxu}$ is uniformly bounded, namely $|\pxu|_x\leq K_0$ with the constant $K_0$ as a common bound, for all $(x,v)\in B$ and $0<|u|< \delta_2$. Hence,
\begin{equation}\label{I22}
	|I_2(x,v,u)|<\left|H_u(x,\pxu,0)-H_u(x,\px,0)\right|+\frac{\vep}{2}, \quad\text{for all~} (x,v)\in B, ~0<|u|< \delta_2.
\end{equation}

On the other hand, it is well known that a continuous function is  uniformly continuous on a compact set, so there exists $\kappa>0$ small enough such that
\begin{equation}\label{I23}
	|H_u(x,p,0)-H_u(x,p',0)|\leq \frac{\vep}{2}, \quad\text{for all~} x\in M,~ |p|_x\leq K_0, ~|p'|_x\leq K_0,~ |p-p'|_x \leq\kappa.
\end{equation}
Due to the upper semi-continuity of the set-valued map $(x,v,u)\mapsto$ $D_2^-L(x,v,u)$,  one can verify that there exists $\delta_3\in(0,\rro)$ small enough such that 
\begin{equation}\label{I24}
	D^-_2L(x,v,u)\subset  D_2^-L(x,v,0) +\kappa, \quad\text{for all~} (x,v)\in B, ~|u|< \delta_3,
\end{equation}
where $ D_2^-L(x,v,0) +\kappa$ stands for the $\kappa$-neighborhood of the set $D_2^-L(x,v,0)$.

Since $\pxu\in D_2^-L(x,v,u)$ and $\px\in D_2^-L(x,v,0)$, we conclude from   \eqref{I22}--\eqref{I24} and item (3) above that
\begin{equation}\label{I25}
	|I_2(x,v,u)|<\frac{\vep}{2}+\frac{\vep}{2}=\vep, \quad\text{for all~} (x,v)\in B, ~0<|u|< \min\{\delta_2,\delta_3\}.
\end{equation}
Finally, by setting $\delta=\min\{\delta_1, \delta_2, \delta_3\}$, inequality \eqref{prf_L_luc} is now a straightforward consequence of \eqref{maxI1I2}, \eqref{I11} and \eqref{I25}. This finishes the proof.
\end{proof}

In the sequel, for each $\lam\in(0,\alo)$, we introduce a Hamiltonian $\HH^\lam(x, p)$ of the form
\begin{align}\label{new_Ham}
	\HH^{\lam}(x, p):=H(x,p, \lam u^\lam(x)),
\end{align}
where $u^\lam$ is the unique solution of \eqref{eq_H_lambda}. 
The new Hamiltonian $\HH^\lam: T^*M\to\R$ is a continuous function defined only on the cotangent bundle. Then  it is easy to verify the following properties:

\begin{Lem}\label{lem_breveHL}
	For $\lam\in(0, \alo)$, the Lagrangian associated to the Hamiltonian $\HH^{\lam}$ is exactly 
    \begin{align}\label{new_Lag}
    	\LL^\lam(x, v):=L(x, v,\lam u^{\lam}(x)).
    \end{align}
Moreover,
    \begin{enumerate}
	\item  $\HH^\lam$ and $\LL^\lam$ are continuous functions, and are both convex and superlinear in the fibers.
	\item $u^\lam$ is also a viscosity solution of the H-J equation $\HH^\lam(x, Du(x))=c(\GG)$. In particular,  the critical value $c(\HH^\lam)=c(\GG)$.
	\item  $\HH^{\lam}$ and  $\LL^{\lam}$, respectively, converge  to $\GG(x,p)$ and $L_{\GG}(x,v)$ uniformly on every compact subset.
	\end{enumerate}
\end{Lem}
This lemma follows immediately from the estimate  $\|\lam u^\lam\|_\infty< \alo\mathbf{C_0}<\rro$ and assumptions {\bf(H1)--\bf(H3)}, so we omit the proof.

Then we can prove the following result:
\begin{Lem}\label{existence_minicurve}
	For each $\lam\in (0,\alo)$ and $x\in M$, there exists a Lipschitz curve $\xilax: (-\infty,0]\to M$ with $\xilax(0)=x$ such that 
	\begin{align}\label{min_curve}
		u^\lam(x)=u^\lam(\xi_x^\lam(-t))+\int^0_{-t} L(\xilax(\tau),\dxilax(\tau),\lam u^\lam(\xilax(\tau)))+c(\GG)\, d\tau,
	\end{align}
	for every $t>0$. In addition, there exists a constant $\sigma_0$, independent of $\lam$ and $x$, such that $\|\dxilax\|_{L^{\infty}}\leq \sigma_0$. In particular, the function $u^\lam\circ\xilax(s)$ is differentiable almost everywhere, and
	\begin{align}\label{du_lam_ae}
		\frac{d}{d s}u^\lam\big(\xilax(s) \big)=&L(\xilax(s),\dxilax(s),\lam u^\lam (\xilax(s)) )+c(\GG),\quad \textup{a.e.} ~ s\in(-\infty,0).
	\end{align}
\end{Lem}
\begin{proof}
	We fix $\lam\in(0,\alpha_0)$ and $x\in M$. From Lemma \ref{lem_breveHL} and item (5) of Proposition \ref{proper_criti_sol}, we have  
	\begin{equation}\label{gjgj1}
		u^\lam(x)=\min_{y\in M} [u^\lam(y)+h^\lam_t(y,x)].
	\end{equation}
Here, $h^\lam_t(x,y)$ stands for the action function of $\LL^\lam$. This means that for each $n\in \Z_+$, there exists a Lipschitz minimizing curve $\xi^\lam_{x,n}:[-n,0]\to M$,  $\xi^\lam_{x,n}(0)=x$ satisfying 
	\begin{equation}\label{cal_1}
u^{\lambda}(\xi^\lam_{x,n}(t_1))=u^{\lambda}(\xi^\lam_{x,n}(t_2))+\int^{t_1}_{t_2}\LL^\lam (\xi^\lam_{x,n}(s),\dot{\xi}^\lam_{x,n}(s))+c(\GG)\ ds,
\end{equation}
for all $-n\leq t_2<t_1\leq 0$. 
Now we claim that  the sequence $\{\xi^\lam_{x,n}\}_n$ is equi-Lipschitz. Indeed, we have already known from Theorem \ref{uni_bound} that $\|u^\lam\|_\infty\leq \mathbf{C_0}$ and $\text{Lip}(u^{\lambda})\leqslant \mathbf{M_0}$, then $\|\lam u^\lam\|_\infty<\alo\mathbf{C_0}<\rro$. Recall that $L(x,v,u)$ is decreasing in $u$ and $L(x,v,\rro): TM\to\R$ is superlinear in $v$, 
this implies that there exists  a constant $\sigma_0>0$, independent of $\lam$ and $x\in M$, such that
\begin{align}
\LL^\lam(x,v)+c(\GG)\geqslant  L(x,v,\rro)+c(\GG)\geqslant
(\mathbf{M_0}+1)|v|_x-\sigma_0.
\end{align}

Next, for each $-t\in (-n,0)$ and $\Delta t>0$, we derive 
\begin{align*}
\mathbf{M_0} d\big(\xi^\lam_{x,n}(-t),\xi^\lam_{x,n}(-t+\Delta t)\big)&\geqslant u^{\lam}(\xi^\lam_{x,n}(-t+\Delta t))-u^{\lam}(\xi^\lam_{x,n}(-t))\\
 &=\int^{-t+\Delta t}_{-t}\LL^\lam(\xi^\lam_{x,n}(s),\dot\xi^\lam_{x,n}(s))+c(\GG)\ ds\\
 &\geqslant \int_{-t}^{-t+\Delta t} (\mathbf{M_0}+1)|\dot{\xi}^\lam_{x,n}(s)|-\sigma_0\,ds\\
 &\geqslant\mathbf{M_0} d\big(\xi^\lam_{x,n}(-t),\xi^\lam_{x,n}(-t+\Delta t)\big)+\int_{-t}^{-t+\Delta t} |\dot{\xi}^\lam_{x,n}(s)|ds-\sigma_0\Delta t,
\end{align*}
hence $$\int_{-t}^{-t+\Delta t} |\dot{\xi}^\lam_{x,n}(s)|\, ds\leqslant \sigma_0\Delta t.$$ 
Dividing  both sides by $\Delta t$ and sending $\Delta t\to 0^+$, we infer 
\begin{equation*}
 \left|\dot{\xi}^\lam_{x,n}(-t)\right|_{\xi^\lam_{x,n}(-t)}\leqslant \sigma_0, \quad \textup{a.e.}~~ -t\in (-n,0).
\end{equation*}
Thus, $\{\xi^\lam_{x,n}\}_n$ is equi-Lipschitz. In particular, by 
the Ascoli-Arzel\`a theorem, there exists a subsequence  $\xi^\lam_{x,n_k}$   converging uniformly, on any compact subinterval of $(-\infty,0]$, to a Lipschitz curve $\xi^\lam_x:(-\infty, 0]\to M$. Obviously,  $\|\dxilax\|_{L^{\infty}}\leq \sigma_0$.
Then it follows  from \eqref{cal_1} and Lemma 
\ref{lower_semi_cont} that 
\begin{align}
		u^\lam(\xi_x^\lam(0))\geq u^\lam(\xi_x^\lam(-t))+\int^0_{-t} \LL^\lam(\xilax(\tau),\dxilax(\tau))+c(\GG)\, d\tau,
	\end{align}
	for every $t>0$. The opposite inequality also holds as a result of \eqref{gjgj1}. This therefore proves equality \eqref{min_curve}.

Finally, by what we have shown above, $u^\lam\circ \xilax (-t)$ is a Lipschitz function of $t$ with $\mathbf{M_0}\sigma_0$ as a Lipschitz bound. Therefore,  one can differentiate \eqref{min_curve} in $t$ and obtain \eqref{du_lam_ae}.
\end{proof}

 Now, for each point $x\in M$ and each $\lam\in(0,\alo)$, we fix a Lipschitz minimizing curve $\xilax:(-\infty, 0]\to M$ with $\xilax(0)=x$ as in Lemma \ref{existence_minicurve}. Then we introduce a probability measure $\tilde\mu^\lam_x$ on $TM$  as follows: 
\begin{equation}\label{def_pro_meas}
	\int_{TM} f(y,v) d\tilde{\mu}^\lam_x(y,v):= \frac{\int_{-\infty}^0    f(\xilax(s),\dxilax(s))\, e^{-\lam\int_0^s L_u(\xilax(\tau),\dxilax(\tau),0)\ d\tau}\, ds}{ \int_{-\infty}^0   e^{-\lam\int_0^s L_u(\xilax(\tau),\dxilax(\tau),0)\ d\tau}\, ds},
\end{equation}
for every $f\in C_c(TM)$. Here we also point out that the function space $C_c(TM)$ can be replaced by $C(TM)$ since  $M$ is compact and $\|\dxilax\|_{L^\infty}\leq \sigma_0$.

\begin{Rem}\label{about_measurable}
It is easy to observe that $\dxilax(s)$ is  a Lebesgue measurable function of $s$. From Lemma \ref{L1_L4} we see that  $L_u(y,v,0)$ $\in C(TM)$ and $L_u(y, v, 0)<0$. Thus, the function
	$$s\mapsto L_u(\xilax(s),\dxilax(s),0)$$ is  Lebesgue measurable.  
	As $\|\dxilax\|_{L^{\infty}}\leq \sigma_0$,  there exist two constants $0<\KK_1\leqslant \KK$, independent of $\lam$ and $x\in M$, such that 
	\begin{equation}\label{Lu_upbound}
			-\KK\leq L_u(\xilax(s),\dxilax(s),0)\leq -\KK_1, \quad\textup{for a.e.~} s\in (-\infty, 0).
	\end{equation}
Thus the integral   \eqref{def_pro_meas} is well defined.
\end{Rem}

In particular, for the discounted case $H(x,p,u)=u+\GG(x,p)$, we have $L_u\equiv -1$, and \eqref{def_pro_meas} is therefore of the form 
\[\int_{TM} f(y,v) d\tilde{\mu}^\lam_x(y,v):= \lam\int_{-\infty}^0e^{\lam s}    f(\xilax(s),\dxilax(s))\, ds.\]
This coincides with  \cite[definition (3.5)]{DFIZ2016}.

Analogous to \cite[Lemma 3.6]{WYZ2020}, the representation \eqref{def_pro_meas} allows us to construct Mather measures from the following approximation scheme:

\begin{Pro}\label{weak_convergence}
Let $x\in M$.
The probability measures $\{\tilde{\mu}^\lam_x: \lam\in(0,\alo)\}$ are relatively compact, in the sense of weak convergence, in the space of  probability measures on $TM$. Moreover, for any  subsequence $\{\tilde{\mu}^{\lambda_i}_x\}_{i}$  weakly converging  to $\tilde{\mu}$ as $\lambda_i\to 0$, $\tilde{\mu}$ is  a Mather measure for the Lagrangian $L_\GG$.
\end{Pro}
\begin{proof}
According to Lemma \ref{existence_minicurve}, the support of each measure $\tilde{\mu}^\lam_x$  is contained in a common compact set $\{(y, v)\in TM: |v|_y\leqslant \sigma_0\}$, which implies the relative compactness. 

Next, without loss of generality, we assume that $\tilde{\mu}^{\lambda}_x$   weakly converges  to a limit measure $\tilde{\mu}$ as $\lam\to 0$. We first show that $\tilde{\mu}$ is a closed measure. For abbreviation, we write  $L_u(s)$ instead of $L_u(\xilax(s),\dxilax(s),0)$ in the following computations.  If  $\phi: M\to \R $ is  a $C^1$ function, then 
\begin{align}
\int_{TM} \du{D\phi, v}_y \,d\tilde{\mu}^{\lam}_x(y,v)=& \frac{\int_{-\infty}^0 e^{-\lam\int_0^s L_u(\tau)\ d\tau}\cdot  \frac{d}{d s}\phi(\xilax(s))\, ds}{\int_{-\infty}^0   e^{-\lam\int_0^s L_u(\tau)\ d\tau}\, ds}\nonumber\\
=&\frac{e^{-\lam\int_0^s L_u(\tau)\ d\tau} \phi(\xilax(s))\Big|^0_{-\infty} -\int_{-\infty}^0 \phi(\xilax(s))\frac{d}{ds}e^{-\lam\int_0^s L_u(\tau)\ d\tau}\, ds  }{\int_{-\infty}^0   e^{-\lam\int_0^s L_u(\tau)\ d\tau}\, ds}\nonumber\\
=&\frac{ \phi(x) -\int_{-\infty}^0 \phi(\xilax(s))\frac{d}{ds}e^{-\lam\int_0^s L_u(\tau)\ d\tau}\, ds }{\int_{-\infty}^0   e^{-\lam\int_0^s L_u(\tau)\ d\tau}\, ds}. \label{close_mea1}
\end{align}
Here, by estimate \eqref{Lu_upbound}, the denominator satisfies $\big|\int_{-\infty}^0   e^{-\lam\int_0^s L_u(\tau)\ d\tau}\, ds|\geq\frac{1}{\lam \KK} $. Besides, 
\begin{equation*}	
\bigg|\int_{-\infty}^0 \phi(\xilax(s))\frac{d}{ds}e^{-\lam\int_0^s L_u(\tau)\ d\tau}\, ds  \bigg|\leqslant\|\phi\|_\infty \int_{-\infty}^0 \frac{d}{ds}e^{-\lam\int_0^s L_u(\tau)\ d\tau}\, ds=\|\phi\|_\infty.
\end{equation*}
Here, we have used that fact that $\frac{d}{ds}e^{-\lam\int_0^s L_u(\tau)\ d\tau}>0$. Consequently, substituting the above inequalities to \eqref{close_mea1} one gets
\begin{equation}
\begin{split}
\int_{TM}\du{D\phi,v}_y\, d\tilde{\mu}(y,v)=\lim\limits_{\lam\to 0} \int_{TM}\du{D\phi, v}_y\, d\tilde{\mu}_x^{\lambda}(y,v)=0.
\end{split}
\end{equation}

Finally, it remains to show that $\int_{TM} L_\GG(y,v)+c(\GG)d\tilde{\mu}(y,v)=0$. Indeed, using \eqref{du_lam_ae} in Lemma \ref{existence_minicurve}, we deduce 
\begin{align*}
&\int_{TM} L_\GG(y,v)+c(\GG)d\tilde{\mu}(y,v)\\
=&\lim\limits_{\lam\to 0} \frac{\int_{-\infty}^0 \left[L(\xilax(s),\dxilax(s),0)+c(\GG)\right]e^{-\lam\int_0^s L_u(\tau)\ d\tau}\, ds}{\int_{-\infty}^0   e^{-\lam\int_0^s L_u(\tau)\ d\tau}\, ds}\\
=& \lim\limits_{\lam\to 0} \frac{\int_{-\infty}^0 \left[\frac{d}{ds} u^\lam(\xilax(s))+L(\xilax(s),\dxilax(s),0)-L(\xilax(s),\dxilax(s),\lam u^\lam(\xilax(s)))\right]e^{-\lam\int_0^s L_u(\tau)d\tau}\, ds}{\int_{-\infty}^0   e^{-\lam\int_0^s L_u(\tau)\ d\tau}\, ds}\\
=& \lim\limits_{\lam\to 0} \mathrm{I}(\lam).
\end{align*}
Since $\|\dxilax\|_{L^{\infty}}\leq \sigma_0$, the locally uniform convergence \eqref{ap_local_uni} implies that there exists $\delta>0$ small enough, such that
$|L(y,v,u)-L(y,v,0)-uL_u(y,v,0)|< |u|$ for all $|u|<\delta$, $y\in M$ and $|v|_y\leq \sigma_0$. Setting $K:=\max_{y\in M, |v|_y\leq \sigma_0}|L_u(y,v,0)|<\infty$, we have $|L(y,v,u)-L(y,v,0)|<|u|(K+1)$ for all $|u|<\delta$, $y\in M$ and $|v|_y\leq \sigma_0$. Therefore, for all $\lam<\delta/\mathbf{C_0}$,  we have
\begin{equation*}
	\left|L(\xilax(s),\dxilax(s),0)-L(\xilax(s),\dxilax(s),\lam u^\lam(\xilax(s)))\right|< \left|\lam u^\lam(\xilax(s))\right|(K+1)<\lam\mathbf{C_0}(K+1),
\end{equation*}
holds for a.e. $s\in(-\infty, 0)$. This implies that for all $\lam$ small enough, 
\begin{align*}
	&\left|\int_{-\infty}^0 \left[L(\xilax(s),\dxilax(s),0)-L(\xilax(s),\dxilax(s),\lam u^\lam(\xilax(s)))\right]e^{-\lam\int_0^s L_u(\tau)d\tau}\, ds\right|\\
	< &\mathbf{C_0}(K+1)\int_{-\infty}^0 \lam e^{-\lam\int_0^s L_u(\tau)d\tau}\, ds\leqslant \mathbf{C_0}(K+1)/\KK_1,
\end{align*}
where $\KK_1$ is given in \eqref{Lu_upbound}. Meanwhile, using integration by parts, we have
\begin{align*}
\left|\int_{-\infty}^0 e^{-\lam\int_0^s L_u(\tau)d\tau} \frac{d}{ds} u^\lam(\xilax(s)) ds\right|=\left|u^\lam(x)- \int_{-\infty}^0 u^\lam(\xilax(s))\frac{d}{ds}e^{-\lam\int_0^s L_u(\tau)d\tau} ds\right|,
\end{align*}
where the right hand side is bounded by $2\|u^\lam\|_\infty\leq 2\mathbf{C_0}$. As a result,  we obtain the estimate 
\[|\mathrm{I}(\lam)|\leqslant \frac{\mathbf{C_0}(K+1)/\KK_1+2\mathbf{C_0}}{1/(\lam \KK)}=\lam \KK\left[\mathbf{C_0}(K+1)/\KK_1+2\mathbf{C_0}\right],\]
which tends to zero
as $\lam\to 0$, and thus $\int_{TM} L_\GG(y,v)+c(\GG)\,d\tilde{\mu}(y,v)=0$. This completes our proof.
\end{proof}

\begin{Lem}\label{new_gen_1}
If the subsequence  $u^{\lambda_k}$ converges uniformly to a continuous function $u^*$ as $\lambda_k\to 0$, then
\begin{align*}
	\int_{TM} L_u(x,v,0)\, u^*(x) \, d\tmu(x,v)\geqslant 0
\end{align*}
for every Mather measure $\tmu\in\wtM(L_\GG)$.
\end{Lem}
\begin{proof}
For each $\lambda_k\in (0,\alo)$, we use the Lagrangian $\LL^{\lam_k}(x,v):TM\to \R$ given in Lemma \ref{lem_breveHL}. Recall that the critical value $c(\HH^{\lam_k})=c(\GG)$, then for each  Mather measure $\tmu\in\wtM(L_\GG)$ which is a closed measure, we deduce from \eqref{mini_pro} that
$
	\int_{TM} \LL^{\lam_k}(x,v)+c(\GG)\,d\tilde{\mu}\geqslant 0,
$
and thus
\begin{align*}
	0\leqslant 
	&\int_{TM}L(x,v,\lam_k u^{\lam_k}(x))-L(x,v,0)\,d\tmu+\int_{TM} L(x,v,0)+c(\GG)\,d\tilde{\mu}\\
	=& \int_{TM} L(x,v,\lam_k u^{\lam_k}(x))-L(x,v,0)\,d\tmu+0.
\end{align*}
This indicates 
\begin{align}\label{frac_1}
	\int_{TM}\frac{L(x,v,\lam_k u^{\lam_k}(x))-L(x,v,0)}{\lam_k }\,d\tilde{\mu}\geqslant 0.
\end{align}

On the other hand, from Proposition \ref{MM_compactness} we know that the support of Mather measure $\tilde{\mu}$ is compact, then it follows from the locally uniform convergence \eqref{ap_local_uni} that, for every $\vep>0$, there is a small number $\delta\in(0,\rro)$, such that 
\begin{equation*}
	|L(x,v,u)-L(x,v,0)-uL_u(x,v,0)|< \vep|u|,\quad \text{for all}~ |u|<\delta,~(x,v)\in \textup{supp}\tilde{\mu}. 
\end{equation*}
Meanwhile, there exists a constant $K=K(\delta)>0$ such that for all $k\geq K$, $\|\lam_k u^{\lam_k}\|_\infty<\delta$. This implies that for  $k\geq K$,
 \[|L(x,v,\lam_k u^{\lam_k}(x))-L(x,v,0)-\lam_k u^{\lam_k}(x)L_u(x,v,0)|< \vep|\lam_k u^{\lam_k}(x)|,\quad \text{for all}~ (x,v)\in \textup{supp}\tilde{\mu}.\]
Dividing both sides by $\lam_k$, and then integrating with respect to $\tilde{\mu}$, one gets that for all $k\geq K$,
 \[\left|\int_{TM}\frac{L(x,v,\lam_k u^{\lam_k}(x))-L(x,v,0)}{\lam_k}\,d\tilde{\mu}-\int_{TM}L_u(x,v,0)u^{\lam_k}(x)\,d\tilde{\mu}\right|< \vep\mathbf{C_0}.\]
Since $u^{\lam_k}$ converges uniformly to $u^*$ and the support of $\tilde{\mu}$ is compact, we have
 \begin{align*}
 	\int_{TM}L_u(x,v,0)u^*(x)\,d\tilde{\mu}&=\lim_{k\to\infty}\int_{TM}L_u(x,v,0)u^{\lam_k}(x)\,d\tilde{\mu}\\
 	&\geq\liminf_{k\to\infty}\int_{TM}\frac{L(x,v,\lam_k u^{\lam_k}(x))-L(x,v,0)}{\lam_k}\,d\tilde{\mu}-\vep\mathbf{C_0}\\
 	&\geq -\vep\mathbf{C_0},
 \end{align*}
where the last inequality comes from  \eqref{frac_1}. Finally, by  sending $\vep\to 0^+$ we conclude that
\begin{align*}
	\int_{TM} L_u(x,v,0)\, u^*(x) \, d\tmu\geqslant 0.
\end{align*} 
\end{proof}

Next, we will prove a lemma which is a generalization of \cite[Lemma 3.7]{DFIZ2016}. It will play a key role in the proof of our main results.
\begin{Lem}\label{new_gen_2}
Let $w$ be a viscosity subsolution	of \eqref{eq_G}. Then for each $\lam\in(0,\alo)$ and $x\in M$,  
\begin{align}
u^\lam(x)	\geqslant & w(x)-\frac{ \int_{TM}   L_u(y,v,0)\, w(y)\, d \tilde{\mu}_x^\lam(y,v)}{\int_{TM} L_u(y,v,0)\,d\tilde{\mu}_x^\lam(y,v)}+R(x,\lam).
\end{align}
Here, the remainder term $R(x,\lam)\longrightarrow 0$ as $\lam $ tends to $0$.
\end{Lem}

\begin{proof}
Let $\lam\in(0,\alo)$.
From \eqref{du_lam_ae} of Lemma \ref{existence_minicurve} we see that for  a.e.  $ s<0$,
\begin{align}
	\frac{d}{d s}u^\lam\big(\xilax(s)\big)=&L(\xilax(s),\dxilax(s),\lam u^\lam(\xilax(s)))+c(\GG)\nonumber\\
	=& L(\xilax(s),\dxilax(s), 0)+\lam  u^\lam(\xilax(s))\,L_u(\xilax(s),\dxilax(s), 0)+\Delta^\lam_x(s)+c(\GG), \label{du_lam_1}
	\end{align}
where 
\begin{align}\label{Del_s}
	\Delta^\lam_x(s)=L(\xilax(s),\dxilax(s),\lam u^\lam(\xilax(s)))-L(\xilax(s),\dxilax(s), 0)-\lam  u^\lam(\xilax(s))\,L_u(\xilax(s),\dxilax(s), 0).
\end{align}
By Lemma \ref{lem_Del_x} which will be proved later, the function $s\mapsto \Delta^\lam_x(s)$ is Lebesgue measurable and bounded. 
For simplicity of notation, we write  $L_u(s)$ instead of $L_u(\xilax(s),$ $\dxilax(s),$ $0)$ in the following proof.

Let  $w(x)$ be any  viscosity subsolution of \eqref{eq_G}. According to Proposition \ref{smooth_sub_approx},
for each $\delta>0$,  there exists a function $w_\delta \in C^\infty(M)$	 such that $\|w_\delta-w\|_\infty\leqslant \delta$,    $\|Dw_\delta\|_\infty\leqslant\textup{Lip}(w)+1$, and 
\begin{align}\label{approxi_G}
	H(x, Dw_\delta(x),0)=\GG(x, Dw_\delta(x))\leqslant c(\GG)+\delta,\quad \textup{for every~} x\in M.
\end{align}
Using Fenchel's inequality,   we deduce from inequalities \eqref{du_lam_1} and \eqref{approxi_G} that
 \begin{align*}
 	\frac{d}{d s}u^\lam(\xilax(s)) \geqslant &\langle Dw_\delta(\xilax(s)), \dxilax(s)\rangle -\delta +\lam  u^\lam(\xilax(s))L_u(s)+\Delta^\lam_x(s)\\
 = & \frac{d}{d s} w_\delta(\xilax(s)) +\lam  u^\lam(\xilax(s))L_u(s)+\Delta^\lam_x(s)-\delta,\quad \textup{for a.e.~}   s<0.
 \end{align*}
Multiplying both sides by $e^{-\lambda\int_0^s L_u(\tau)\, d\tau}$ and rearranging the terms, one gets
\begin{align*}
\frac{d}{d s}\left(u^\lam(\xilax(s))\, e^{-\lam\int_0^s L_u(\tau)\, d\tau} \right)	\geqslant \left( \frac{d}{d s} w_\delta(\xilax(s)) +\Delta^\lam_x(s)-\delta\right)e^{-\lam\int_0^s L_u(\tau)\, d\tau}, ~\textup{for a.e.~}   s<0.
\end{align*}
Integrating this inequality with respect to $s$ over the interval $(-\infty,0]$, we obtain
\begin{align}\label{gdht}
u^\lam(x)	\geqslant &  w_\delta(x) -\int_{-\infty}^0   w_\delta(\xilax(s)) \frac{d}{d s} \Big(e^{-\lam\int_0^s L_u(\tau)\, d\tau}\Big)\,ds+\int_{-\infty}^0(\Delta^\lam_x(s) -\delta) e^{-\lam\int_0^s L_u(\tau)\, d\tau}\, ds. 
\end{align}
Here, we have used integration by parts.  Observe that
\begin{align*}
	\left|\int_{-\infty}^0  \big( w_\delta(\xilax(s))-w(\xilax(s))\big) \frac{d}{d s}\Big(e^{-\lam\int_0^s L_u(\tau)\, d\tau}\Big)\,ds\right|\leqslant \delta,
\end{align*}
then we derive from \eqref{gdht} that
\begin{align*}
u^\lam(x)	\geqslant & w_\delta(x) -\int_{-\infty}^0   w(\xilax(s))\frac{d}{d s} \Big(e^{-\lam\int_0^s L_u(\tau)\, d\tau}\Big)ds-\delta+\int_{-\infty}^0(\Delta^\lam_x(s) -\delta)e^{-\lam\int_0^s L_u(\tau)\, d\tau}\, ds.
\end{align*}
Sending  $\delta\to 0$,  we infer 
\begin{align}
u^\lam(x)	\geqslant & w(x)-\int_{-\infty}^0   w(\xilax(s))\frac{d}{d s} \Big(e^{-\lam\int_0^s L_u(\tau)\, d\tau}\Big)\,ds +\underbrace{\int_{-\infty}^0\Delta^\lam_x(s) e^{-\lam\int_0^s L_u(\tau)\, d\tau}\, ds}_{R(x,\lam)}\nonumber\\
=& w(x)+\lam \int^0_{-\infty}    w(\xilax(s))L_u(s)e^{-\lam\int_0^s L_u(\tau)\, d\tau} \,ds+R(x,\lam)\nonumber\\
=& w(x)+\lam \left(\int_{-\infty}^0  e^{-\lam\int_0^s L_u(\tau)\,d\tau}\,ds \right)  \cdot \int_{TM}   w(y) L_u(y,v,0) \,d\tilde{\mu}_x^\lam(y,v)+R(x,\lam).\nonumber
\end{align}

Next, we will show that \[\lam \int_{-\infty}^0   e^{-\lam\int_0^s L_u(\tau)\, d\tau}\, ds=-1/\int_{TM} L_u(y,v,0)\,d\tilde{\mu}_x^\lam(y,v).\] Indeed, according to the definition of $\tilde{\mu}_x^\lam $,
\begin{align*}
	\int_{TM} L_u(y,v,0)\,d\tilde{\mu}_x^\lam(y,v)=\frac{\int_{-\infty}^0    L_u(s) e^{-\lam\int_0^s L_u(\tau)\ d\tau}\, ds}{ \int_{-\infty}^0   e^{-\lam\int_0^s L_u(\tau)\ d\tau}\, ds}=&\frac{ \int_{-\infty}^0 \frac{d}{ds}e^{-\lam\int_0^s L_u(\tau)\, d\tau}\,ds}{-\lam \int_{-\infty}^0   e^{-\lam\int_0^s L_u(\tau)\, d\tau}\, ds}\\
	=&\frac{ 1}{-\lam \int_{-\infty}^0   e^{-\lam\int_0^s L_u(\tau)\, d\tau}\, ds}.
\end{align*}

Now, it only remains to show that $R(x,\lam)\to 0$ as $\lam\to 0$. In fact, for every $\vep>0$, by Lemma \ref{lem_Del_x}, there exists a constant $\lam_\vep>0$ small enough such that	 for each $\lam\in (0,\lam_\vep)$,\begin{equation}
		|\Delta^\lam_x(s)|\leq \vep \lam, \quad \textup{a.e.~} s\in(-\infty,0).
	\end{equation} 
	 This leads to
	 \begin{equation}
	 	|R(x,\lam)|\leqslant \vep \int_{-\infty}^0 \lam e^{-\lam\int_0^s L_u(\tau)\, d\tau}\, ds\leqslant\frac{\vep}{\KK_1},\quad \textup{for all~}\lam\in (0,\lam_\vep),
	 \end{equation} 
where the last inequality follows from \eqref{Lu_upbound}.	Therefore, 
by the arbitrariness of $\vep$, we have $\lim_{\lam\to 0}R(x,\lam)$ $=0 $.
This finishes the proof.
\end{proof}

To complete the proof of Lemma \ref{new_gen_2}, it remains to show the following result:
\begin{Lem}\label{lem_Del_x}
	For each $\lam\in(0,\alo)$ and $x\in M$, let $\Delta_x^\lam(s):(-\infty,0]\to\R$ be the function defined in \eqref{Del_s}. Then $\Delta_x^\lam(s)$ is a bounded Lebesgue measurable function. Moreover, for each $\vep>0$, there exists a small number $\lam_\vep>0$, independent of $x$, such that	for each $\lam\in (0,\lam_\vep)$,
	\begin{equation}
		|\Delta^\lam_x(s)|\leq \vep \lam, \quad \textup{a.e.~} s\in(-\infty,0).
	\end{equation} 
\end{Lem}
\begin{proof}
If we set $f(y,v,u):=L(y,v,u)-L(y,v,0)-uL_u(y,v,0)$, then Lemma \ref{L1_L4} implies that $f$ is continuous on $TM\times(-\rro,+\infty)$, and hence $\Delta^\lam_x(s)=f\circ \big(\xilax(s), \dxilax(s),\lam u^\lam(\xilax(s))\big)$ is Lebesgue measurable and bounded. 

 By Lemma \ref{existence_minicurve} and Proposition \ref{uni_bound}, $\|\dxilax\|_{L^\infty}\leqslant \sigma_0$ and    $\|u^\lam\|_\infty\leqslant \mathbf{C_0}$, 
uniformly in  $\lam\in (0,\alo)$ and $x\in M$.  For every $\vep>0$, we set $\vep_0=\vep/\mathbf{C_0}$, then it follows from the locally uniform convergence  \eqref{ap_local_uni} that there exists a small number
$\kappa=\kappa(\vep, \sigma_0)\in (0,\rro)$ such that
\begin{equation*}
	\Big|\frac{L(y,v,u)-L(y,v,0)}{u}-L_u(y,v,0)\Big|\leqslant\vep_0,\quad \text{uniformly in~} y\in M, |v|_y\leqslant \sigma_0, 0<|u|\leqslant \kappa.
\end{equation*} 
Now we have
\begin{equation}\label{eps_0}
	|f(y,v,u)|\leqslant \vep_0|u|, \quad \text{uniformly in~} y\in M, |v|_y\leqslant \sigma_0, |u|\leqslant \kappa.
\end{equation} 
Therefore, by setting $\lam_\vep:=\kappa/\mathbf{C_0}$, inequality \eqref{eps_0} indicates that for every $0<\lam<\lam_\vep$, 
\begin{equation*}
		|\Delta^\lam_x(s)|=|f\big(\xilax(s), \dxilax(s),\lam u^\lam(\xilax(s))\big)|\leq \vep_0 \lam|u^\lam(\xilax(s))|\leq \vep\lam, \quad \textup{a.e.~} s\in(-\infty,0).
	\end{equation*}
\end{proof}

\subsection{Convergence}

Inspired by Lemma \ref{new_gen_1}, we denote by $\FF$  the set of all viscosity subsolutions $w$ of \eqref{eq_G} satisfying
\begin{align}\label{set_FF}
	\int_{TM} L_u(x,v,0)w(x)\,d\tilde{\mu}(x,v)\geqslant 0,\quad \textup{for all}~\tilde{\mu}\in\wtM (L_\GG).
\end{align}
We claim that the family $\FF$ is uniformly bounded above, i.e.
\begin{equation*}
	\sup\{w(x)~:~x\in M, w\in\FF  \}<+\infty.
\end{equation*}
 Indeed, for each $w\in\FF$,  it follows from  \eqref{set_FF} and $L_u(x,v,0)<0$ that $\min_{x\in M} w(x)\leqslant 0$. Besides, by Proposition \ref{sublip} all viscosity subsolutions of \eqref{eq_G} are equi-Lipschitz continuous with $\rho_c$ as a Lipschitz bound, so $\max_{x\in M} w(x)\leq \max_{x\in M} w(x)-\min_{x\in M} w(x) \leq \rho_c \cdot \textup{diam} (M)$. 
 
 This allows us to  define a function
$u^0:M\to \R$ by 
\begin{equation}\label{first_charac}
	u^0(x):=\sup_{w\in\FF} w(x).
\end{equation}
Obviously, $u^0$ is Lipschitz continuous on $M$.

Now we are ready to prove our convergence result:
\begin{proof}[Proof of Theorem \ref{mainresult1}:]
As the supremum of a family of viscosity subsolutions, the function $u^0(x)$,  defined in \eqref{first_charac},  is itself a viscosity subsolution of \eqref{eq_G}. We will show $\lim_{\lam\to 0}u^\lam=u^0$. 

Indeed, 
 by Proposition \ref{uni_bound}, the family $\{u^\lam\}_{\lam\in(0,\alo)}$ is equi-Lipschitz continuous and equi-bounded. Hence, by the Ascoli-Arzel\`a theorem it only remains to prove that any convergent subsequence  has the same limit $u^0$.

If $\{u^{\lam_k}\}_k$ is a subsequence converging uniformly to some continuous function $u:M\to\R $ as $\lam_k$ tends to zero, it is clear from Lemma \ref{new_gen_1} that this limit function $u\in \FF$, so  $u\leq u^0$.
Now, we only need to show $u\geq u^0$. 
For every $x\in M$,  it follows from Lemma \ref{new_gen_2} that for each $w\in \FF$,
\begin{align}\label{geq_1}
u^{\lam_k}(x)	\geqslant & w(x)-\frac{ \int_{TM}  L_u(y,v,0)\, w(y)\, d \tilde{\mu}_x^{\lam_k}(y,v)}{\int_{TM} L_u(y,v,0)\,d\tilde{\mu}_x^{\lam_k}(y,v)}+R(x,\lam_k).
\end{align}
By Proposition \ref{weak_convergence}, extracting a subsequence if necessary, we can assume that $\tilde{\mu}_x^{\lam_k}$ converges weakly to a Mather measure $\tilde{\mu}$. Therefore, sending $\lam_k\to 0$, inequality \eqref{geq_1} leads to 
\begin{align*}
u(x)	\geqslant & w(x)-\frac{ \int_{TM}  L_u(y,v,0)\, w(y)\, d \tilde{\mu}(y,v)}{\int_{TM} L_u(y,v,0)\,d\tilde{\mu}(y,v)}\geqslant w(x),
\end{align*}
where the last inequality follows from $L_u(y,v,0)<0$ and $w\in\FF$. Therefore,  $u(x)\geqslant \sup_{w\in\FF} w(x)=u^0(x)$ for every $x\in M$. This completes the proof.
\end{proof}

\subsection{Characterization of the limit}
\begin{proof}[Proof of Theorem \ref{mainresult2}:]
By what we have shown above, we already obtain the first representation formula \[u^0(x)=\sup_{w\in\FF} w(x).\]  It remains to prove the second one that is characterized in terms of Peierls barrier. 
  
By Proposition \ref{proper_criti_sol},  $u^0(x)\leqslant u^0(y)+h(y,x)$ for all $x, y\in M$. Also, we recall that each Mather measure has compact support (see Proposition \ref{MM_compactness}). Then, multiplying both sides by $L_u(y,v,0)$ which is negative, and integrating  with respect to $y$, we derive 
\begin{equation*}
\begin{split}
	u^0(x)\int_{TM}L_u(y,v,0)d\tilde{\mu}(y,v)&\geqslant \int_{TM}L_u(y,v,0)u^0(y)d\tilde{\mu}(y,v)+ \int_{TM}L_u(y,v,0)h(y,x)d\tilde{\mu}(y,v)\\
	&\geq \int_{TM}L_u(y,v,0)h(y,x)d\tilde{\mu}(y,v),		
\end{split}
\end{equation*}
for every Mather measure $\tilde{\mu}\in\wtM(L_\GG) $, and the last inequality follows from Lemma \ref{new_gen_1}.
Thus,
\begin{align*}
	u^0(x)\leqslant \frac{\int_{TM}L_u(y,v,0)h(y,x)d\tilde{\mu}(y,v)}{\int_{TM}L_u(y,v,0)d\tilde{\mu}(y,v)},\quad \textup{for every~} \tilde{\mu}\in\wtM(L_\GG).
\end{align*}
As a result,
\begin{equation}\label{REP_1}
		u^0(x)\leq\inf_{\tilde{\mu}\in \wtM(L_\GG)}\frac{\int_{TM}L_u(y,v,0)h(y,x)d\tilde{\mu}(y,v)}{\int_{TM}L_u(y,v,0)d\tilde{\mu}(y,v)}.
	\end{equation}

To prove the opposite inequality, we first consider $x\in \cA$ with $\cA$ the projected Aubry set of $L_\GG$. By item (2) of Proposition \ref{proper_criti_sol}, the function $y\mapsto -h(y,x)$ is a viscosity subsolution of \eqref{eq_G}. This in turn gives, thanks to Lemma \ref{new_gen_2},
\begin{align*}
u^\lam(x)	\geqslant & -h(x,x)+\frac{ \int_{TM}   L_u(y,v,0) h(y,x)\, d \tilde{\mu}_x^\lam(y,v)}{\int_{TM} L_u(y,v,0)\,d\tilde{\mu}_x^\lam(y,v)}+R(x,\lam)\\
=&\frac{ \int_{TM}   L_u(y,v,0) h(y,x)\, d \tilde{\mu}_x^\lam(y,v)}{\int_{TM} L_u(y,v,0)\,d\tilde{\mu}_x^\lam(y,v)}+R(x,\lam).
\end{align*}
The last equality follows from the fact that $h(x,x)=0$ for $x\in\cA$.
By Proposition \ref{weak_convergence}, extracting a subsequence if necessary, we can assume that $\tilde{\mu}_x^{\lam}$ converges weakly to a Mather measure $\tilde{\nu}$. 
Now, sending $\lam\to 0$ we infer that
\begin{align*}
u^0(x)	\geqslant \frac{ \int_{TM}   L_u(y,v,0) h(y,x)\, d \tilde{\nu}(y,v)}{\int_{TM} L_u(y,v,0)\,d\tilde{\nu}(y,v)},
\end{align*}
which also gives rise to
\begin{align}\label{REP_2}
u^0(x)\geq\inf_{\tilde{\mu}\in \wtM(L_\GG)}\frac{\int_{TM}L_u(y,v,0)h(y,x)d\tilde{\mu}(y,v)}{\int_{TM}L_u(y,v,0)d\tilde{\mu}(y,v)},
\end{align}
for all $x\in \cA$. Note that the function on the right hand side of \eqref{REP_2} is a viscosity subsolution of \eqref{eq_G}, then according to item (4) of Proposition \ref{proper_criti_sol}, inequality \eqref{REP_2} holds for all $x\in M$. This completes the proof.
\end{proof}

\medskip

Finally, as a concluding remark, we return to the system studied in 
 Proposition \ref{uni_const}. Actually, in the case where constant functions are critical subsolutions, the limit solution $u^0$ has a simpler representation formula: 

\begin{Pro}
Let $H$  satisfy  assumptions {\bf(H1)}, {\bf(H2*)} and {\bf(H3)}. Suppose that constant functions are subsolutions of equation \eqref{eq_G}, or equivalently $\GG(x, 0)\leq c(\GG)$ on $M$. Then
\begin{equation}
	u^0(x)=\min_{y\in\cA} h(y,x),
\end{equation}
where $h(y,x)$ and $\cA$ are, respectively, the Peierls barrier and the projected Aubry set of  $L_\GG$.
\end{Pro}

\begin{proof}
Let us set
$
		\widehat{u}(x):=\min_{y\in\mathcal{A}} h(y,x).
$
Since constant functions are subsolutions of \eqref{eq_G},  by item (3) of Proposition \ref{proper_criti_sol} we infer that $h(y,x)\geqslant 0$ for all $x$, $y$, and thus $\widehat{u}\geq 0$. 
In view of item (2) of Proposition \ref{proper_criti_sol}, we see  
that $\widehat{u}$ is a minimum of a family of equi-Lipschitz solutions to \eqref{eq_G}. Then it is well known that
$\widehat{u}$  is itself a viscosity solution of \eqref{eq_G}. 

Next, observe that $\widehat{u}\geq 0$ 
is a  supersolution of \eqref{eq_H_lambda}. Then, by the comparison principle and Proposition \ref{uni_const} we get
$
	0\leq u^\lambda\leq \widehat{u}.
$
In particular, since $\widehat{u}(x)=0$ for $x\in \mathcal{A}$, we have
\begin{equation*}
	u^\lambda(x)=0,\quad \forall~x\in \mathcal{A}.
\end{equation*}
This implies that the limit solution  $u^0(x)=\lim_{\lam\to 0^+}u^\lam(x)=0$ for all $x\in\mathcal{A}$. Therefore, item (4) of Proposition \ref{proper_criti_sol} ensures  that $u^0=\widehat u$ on $M$. 
\end{proof}

\appendix
\section{}
Here, we provide an equivalent description for assumption {\bf(H2)} provided that {\bf(H1)} and {\bf(H3)} hold.
\begin{Pro}\label{equiv_def_H2}
	If  $H\in C(T^*M\times\R) $ satisfies the convexity assumption {\bf(H1)} and the monotonicity assumption {\bf(H3)}, then the following properties are equivalent:
	\begin{enumerate}[(1)]
		\item There is a constant $\rro>0$ such that  $\lim_{|p|_x\to \infty}$ $H(x, p, -\rro)$ $=+\infty$, uniformly in  $x\in M.$
	\item 	$ \lim_{|p|_x\to \infty} H(x, p, 0)=+\infty$, uniformly in  $x\in M.$
	\end{enumerate}
\end{Pro}
\begin{proof}
By the monotonicity {\bf(H3)}, the implication $(1)\Rightarrow (2)$ is  obvious.

Now, we show that $(2)$ implies $(1)$. To this end,  only assumption {\bf(H1)} will be used. Indeed, by the coercivity of $H(x, p, 0)$ and the compactness of $M$, there is a constant $r>0$  such that 
\begin{equation*}
	H(x, p,0) > \max_{x\in M}H(x,0,0)+2,\quad \text{for all~} (x,p)\in \partial B_r,
\end{equation*}
where $\partial B_r$ denotes the compact set $\{(x,p)\in T^*M : x\in M, |p|_x=r\}.$
By continuity and compactness, there exists a small constant $\rro>0$ such that 
\begin{equation}\label{gjh0}
	H(x, p, -\rro) > \max_{x\in M}H(x,0,-\rro) +1 ,\quad \text{for all~} (x,p)\in \partial B_r.
\end{equation}

Next, for each $(x, p)$ with $|p|_x>r$, we pick a point $z_{x,p}:=\frac{r}{|p|_x} p$   on $\partial B_r$. Then, the  convexity assumption {\bf(H1)} implies 
\begin{equation}\label{gjh1}
	H(x, z_{x,p},-\rro)\leq (1-\frac{r}{|p|_x})H(x,0,-\rro)+\frac{r}{|p|_x} H(x, p, -\rro),\quad \text{for all~} |p|_x>r.
\end{equation} 
 Since $z_{x,p}\in\partial B_r$,  we infer from \eqref{gjh0}-\eqref{gjh1} that
\begin{equation*}
	H(x, p, -\rro)\geq \frac{H(x, z_{x,p},-\rro)-H(x, 0,-\rro)}{r/|p|_x} +H(x, 0,-\rro)\geq \frac{1}{r}|p|_x+ H(x, 0,-\rro),
\end{equation*} 
for all $|p|_x>r$. 
This implies $\lim_{|p|_x\to \infty} H(x, p, -\rro)=+\infty$, uniformly in  $x\in M.$
\end{proof}

\noindent\textbf{Acknowledgments.} 
This work was initiated while I was visiting 
the Georgia Institute of Technology (December 2018--March 2019),
 whose warm hospitality is gratefully acknowledged. It was completed
while I was working at the University of Padova.
I am particularly grateful to Albert Fathi for many stimulating discussions during my stay at Georgia Tech, 
 and for his careful reading of the manuscript and valuable comments. I wish to thank Wei Cheng, Jun Yan and Jianlu Zhang for some helpful discussions.  I also would like to thank Gabriella Pinzari for her support. This research is funded by the ERC Project 677793 StableChaoticPlanetM.


\begin{thebibliography}{WWY19b}

\bibitem[AAAIY16]{AAIY2016}
Eman~S. Al-Aidarous, Ebraheem~O. Alzahrani, Hitoshi Ishii, and Arshad M.~M.
  Younas.
\newblock A convergence result for the ergodic problem for {H}amilton-{J}acobi
  equations with {N}eumann-type boundary conditions.
\newblock {\em Proc. Roy. Soc. Edinburgh Sect. A}, 146(2):225--242, 2016.

\bibitem[Bar94]{Barles_book}
Guy Barles.
\newblock {\em Solutions de viscosit{\'e} des {\'e}quations de
  {H}amilton-{J}acobi}, volume~17 of {\em Math{\'e}matiques \& Applications}.
\newblock Springer-Verlag, Paris, 1994.

\bibitem[BGH98]{Buttazzo_Giaquinta_Hildebrandt_book}
Giuseppe Buttazzo, Mariano Giaquinta, and Stefan Hildebrandt.
\newblock {\em One-dimensional variational problems}, volume~15 of {\em Oxford
  Lecture Series in Mathematics and its Applications}.
\newblock The Clarendon Press, Oxford University Press, New York, 1998.
\newblock An introduction.

\bibitem[CCIZ19]{CCIZ2019}
Qinbo Chen, Wei Cheng, Hitoshi Ishii, and Kai Zhao.
\newblock Vanishing contact structure problem and convergence of the viscosity
  solutions.
\newblock {\em Comm. Partial Differential Equations}, 44(9):801--836, 2019.

\bibitem[CCJ{\etalchar{+}}20]{Cannarsa2019herglotz}
Piermarco Cannarsa, Wei Cheng, Liang Jin, Kaizhi Wang, and Jun Yan.
\newblock Herglotz' variational principle and {L}ax-{O}leinik evolution.
\newblock {\em J. Math. Pures Appl. (9)}, 141:99--136, 2020.

\bibitem[CCWY19]{CCWY2019}
Piermarco Cannarsa, Wei Cheng, Kaizhi Wang, and Jun Yan.
\newblock Herglotz' generalized variational principle and contact type
  {H}amilton-{J}acobi equations.
\newblock In {\em Trends in control theory and partial differential equations},
  volume~32 of {\em Springer INdAM Ser.}, pages 39--67. Springer, Cham, 2019.

\bibitem[CP19]{Cardaliaguet_Porretta2019}
Pierre Cardaliaguet and Alessio Porretta.
\newblock Long time behavior of the master equation in mean field game theory.
\newblock {\em Anal. PDE}, 12(6):1397--1453, 2019.

\bibitem[CS04]{Cannarsa_Sinestrari_book}
Piermarco Cannarsa and Carlo Sinestrari.
\newblock {\em Semiconcave functions, {H}amilton-{J}acobi equations, and
  optimal control}, volume~58 of {\em Progress in Nonlinear Differential
  Equations and their Applications}.
\newblock Birkh{\"a}user Boston, Inc., Boston, MA, 2004.

\bibitem[DFIZ16a]{DFIZ_Math_Z_2016}
Andrea Davini, Albert Fathi, Renato Iturriaga, and Maxime Zavidovique.
\newblock Convergence of the solutions of the discounted equation: the discrete
  case.
\newblock {\em Math. Z.}, 284(3-4):1021--1034, 2016.

\bibitem[DFIZ16b]{DFIZ2016}
Andrea Davini, Albert Fathi, Renato Iturriaga, and Maxime Zavidovique.
\newblock Convergence of the solutions of the discounted {H}amilton-{J}acobi
  equation: convergence of the discounted solutions.
\newblock {\em Invent. Math.}, 206(1):29--55, 2016.

\bibitem[DW20]{Davini_Wang2020}
Andrea Davini and Lin Wang.
\newblock On the vanishing discount problem from the negative direction.
\newblock {\em arXiv:2007.12458}, to appear in Discrete Contin. Dyn. Syst.
  2020.

\bibitem[DZ37]{Davini_Zavidovique2017}
Andrea Davini and Maxime Zavidovique.
\newblock Convergence of the solutions of discounted {H}amilton--{J}acobi
  systems.
\newblock {\em Adv. Calc. Var.}, Online publication (2019),
  https://doi.org/10.1515/acv-2018-0037.

\bibitem[DZ13]{Davini_Zavidovique2013}
Andrea Davini and Maxime Zavidovique.
\newblock Weak {KAM} theory for nonregular commuting {H}amiltonians.
\newblock {\em Discrete Contin. Dyn. Syst. Ser. B}, 18(1):57--94, 2013.

\bibitem[Fat]{Fathi_book}
Albert Fathi.
\newblock {W}eak {KAM} theorem in {L}agrangian dynamics, preliminary version 10. 
\newblock unpublished book, 2008.

\bibitem[Fat12]{Fathi2012}
Albert Fathi.
\newblock Weak {KAM} from a {PDE} point of view: viscosity solutions of the
  {H}amilton-{J}acobi equation and {A}ubry set.
\newblock {\em Proc. Roy. Soc. Edinburgh Sect. A}, 142(6):1193--1236, 2012.

\bibitem[FS05]{Fathi_Siconolfi2005}
Albert Fathi and Antonio Siconolfi.
\newblock P{DE} aspects of {A}ubry-{M}ather theory for quasiconvex
  {H}amiltonians.
\newblock {\em Calc. Var. Partial Differential Equations}, 22(2):185--228,
  2005.

\bibitem[GMT18]{Gomes_Mitake_Tran2018}
Diogo~A. Gomes, Hiroyoshi Mitake, and Hung~V. Tran.
\newblock The selection problem for discounted {H}amilton-{J}acobi equations:
  some non-convex cases.
\newblock {\em J. Math. Soc. Japan}, 70(1):345--364, 2018.

\bibitem[GMT19]{gomes2019selection}
Diogo~Aguiar Gomes, Hiroyoshi Mitake, and Kengo Terai.
\newblock The selection problem for some first-order stationary {M}ean-field
  games.
\newblock {\em Networks \& Heterogeneous Media}, 15(4):681--710, 2020.

\bibitem[Gom08]{Gomes2008}
Diogo~Aguiar Gomes.
\newblock Generalized {M}ather problem and selection principles for viscosity
  solutions and {M}ather measures.
\newblock {\em Adv. Calc. Var.}, 1(3):291--307, 2008.

\bibitem[IJ20]{Ishii_Jin2020}
Hitoshi Ishii and Liang Jin.
\newblock The vanishing discount problem for monotone systems of
  {H}amilton-{J}acobi equations. {P}art 2: Nonlinear coupling.
\newblock {\em Calc. Var. Partial Differential Equations}, 59, 2020.

\bibitem[IMT17a]{Ishii_Mitake_Tran2017_1}
Hitoshi Ishii, Hiroyoshi Mitake, and Hung~V. Tran.
\newblock The vanishing discount problem and viscosity {M}ather measures.
  {P}art 1: {T}he problem on a torus.
\newblock {\em J. Math. Pures Appl. (9)}, 108(2):125--149, 2017.

\bibitem[IMT17b]{Ishii_Mitake_Tran2017_2}
Hitoshi Ishii, Hiroyoshi Mitake, and Hung~V. Tran.
\newblock The vanishing discount problem and viscosity {M}ather measures.
  {P}art 2: {B}oundary value problems.
\newblock {\em J. Math. Pures Appl. (9)}, 108(3):261--305, 2017.

\bibitem[IS20]{Ishii_Siconolfi2020}
Hitoshi Ishii and Antonio Siconolfi.
\newblock The vanishing discount problem for {H}amiltion-{J}acobi equations in
  the {E}uclidean space.
\newblock {\em Comm. Partial Differential Equations}, 45(6):525--560, 2020.

\bibitem[Ish87]{Ishii1987}
Hitoshi Ishii.
\newblock Perron's method for {H}amilton-{J}acobi equations.
\newblock {\em Duke Math. J.}, 55(2):369--384, 1987.

\bibitem[Ish08]{Ishii2008}
Hitoshi Ishii.
\newblock Asymptotic solutions for large time of {H}amilton-{J}acobi equations
  in {E}uclidean {$n$} space.
\newblock {\em Ann. Inst. H. Poincar{\'e} Anal. Non Lin{\'e}aire},
  25(2):231--266, 2008.

\bibitem[Ish19]{Ishii2019vanishing}
Hitoshi Ishii.
\newblock The vanishing discount problem for monotone systems of
  {H}amilton-{J}acobi equations. {P}art 1: linear coupling.
\newblock {\em arXiv:1903.00244}, 2019.

\bibitem[ISM11]{Iturriaga_Sanchez-Morgado2011}
Renato Iturriaga and H{\'e}ctor S{\'a}nchez-Morgado.
\newblock Limit of the infinite horizon discounted {H}amilton-{J}acobi
  equation.
\newblock {\em Discrete Contin. Dyn. Syst. Ser. B}, 15(3):623--635, 2011.

\bibitem[LPV87]{Lions_Papanicolaou_Varadhan1987}
P.-L. Lions, G.~Papanicolaou, and S.R.S Varadhan.
\newblock Homogenization of {H}amilton-{J}acobi equations.
\newblock  unpublished work, 1987.

\bibitem[Mat91]{Mather1991}
John~N. Mather.
\newblock Action minimizing invariant measures for positive definite
  {L}agrangian systems.
\newblock {\em Math. Z.}, 207(2):169--207, 1991.

\bibitem[Mat93]{Mather1993}
John~N. Mather.
\newblock Variational construction of connecting orbits.
\newblock {\em Ann. Inst. Fourier (Grenoble)}, 43(5):1349--1386, 1993.

\bibitem[Mn96]{Mane1996}
Ricardo Ma\~n{\'e}.
\newblock Generic properties and problems of minimizing measures of
  {L}agrangian systems.
\newblock {\em Nonlinearity}, 9(2):273--310, 1996.

\bibitem[MS17]{Maro_Sorrentino2017}
Stefano Mar{\`o} and Alfonso Sorrentino.
\newblock Aubry-{M}ather theory for conformally symplectic systems.
\newblock {\em Comm. Math. Phys.}, 354(2):775--808, 2017.

\bibitem[MT17]{Mitake_Tran2017}
Hiroyoshi Mitake and Hung~V. Tran.
\newblock Selection problems for a discount degenerate viscous
  {H}amilton-{J}acobi equation.
\newblock {\em Adv. Math.}, 306:684--703, 2017.

\bibitem[SWY16]{Su_Wang_Yan2016}
Xifeng Su, Lin Wang, and Jun Yan.
\newblock Weak {KAM} theory for {H}amilton-{J}acobi equations depending on
  unknown functions.
\newblock {\em Discrete Contin. Dyn. Syst.}, 36(11):6487--6522, 2016.

\bibitem[WWY19a]{Wang_Wang_Yan2019_Aubry}
Kaizhi Wang, Lin Wang, and Jun Yan.
\newblock Aubry-{M}ather theory for contact {H}amiltonian systems.
\newblock {\em Comm. Math. Phys.}, 366(3):981--1023, 2019.

\bibitem[WWY19b]{Wang_Wang_Yan2019_Variational}
Kaizhi Wang, Lin Wang, and Jun Yan.
\newblock Variational principle for contact {H}amiltonian systems and its
  applications.
\newblock {\em J. Math. Pures Appl. (9)}, 123:167--200, 2019.

\bibitem[WYZ20]{WYZ2020}
Yanan Wang, Jun Yan, and Jianlu Zhang.
\newblock Convergence of viscosity solutions of generalized contact
  {H}amilton-{J}acobi equations.
\newblock {\em arXiv:2004.12269}, 2020.

\bibitem[Zav20]{Zavidovique2020}
Maxime Zavidovique.
\newblock Convergence of solutions for some degenerate discounted
  {H}amilton--{J}acobi equations.
\newblock {\em arXiv:2006.00779}, 2020.

\bibitem[Zil19]{Ziliotto_counterexample_2019}
Bruno Ziliotto.
\newblock Convergence of the solutions of the discounted {H}amilton-{J}acobi
  equation: a counterexample.
\newblock {\em J. Math. Pures Appl. (9)}, 128:330--338, 2019.

\end{thebibliography}
\bibliographystyle{alpha}
\newcommand{\etalchar}[1]{$^{#1}$}

\end{document}